\documentclass[12pt]{article}

\usepackage{amsmath,enumerate,amsfonts,amssymb,color,graphicx,amsthm}

\def\RR{{\mathbb R}}

\def\Snn{{\mathcal{S}^{n \times n}}}

\newcounter{marnote}

\begin{document}
\newtheorem{thm}{Theorem}[section]
\newtheorem{Def}[thm]{Definition}
\newtheorem{lem}[thm]{Lemma}
\newtheorem{rem}[thm]{Remark}
\newtheorem{question}[thm]{Question}
\newtheorem{prop}[thm]{Proposition}
\newtheorem{cor}[thm]{Corollary}
\newtheorem{example}[thm]{Example}

\title{Comparison principles and Lipschitz regularity for some
 nonlinear degenerate elliptic equations}

\author{YanYan Li \footnote{School of Mathematical Sciences, Beijing Normal University, Beijing 100875, China and Department of Mathematics, Rutgers University, 110 Frelinghuysen Rd, Piscataway, NJ 08854, USA. Email: yyli@math.rutgers.edu.}~\quad Luc Nguyen \footnote{Mathematical Insitute and St Edmund Hall, University of Oxford, Andrew Wiles Building, Radcliffe Observatory Quarter, Woodstock Road, Oxford OX2 6GG, UK. Email: luc.nguyen@maths.ox.ac.uk.}~\quad Bo Wang \footnote{
School of Mathematics and Statistics, Beijing Institute of Technology, Beijing 100081, China. Email: wangbo89630@bit.edu.cn.}}

\date{}

\maketitle

\begin{abstract}
We establish interior Lipschitz regularity for continuous viscosity solutions of fully nonlinear, conformally invariant, degenerate elliptic equations. As a by-product of our method, we also prove a weak form of the strong comparison principle, which we refer to as the principle of propagation of touching points, for operators of the form $\nabla^2 \psi + L(x,\psi,\nabla \psi)$ which are non-decreasing in $\psi$. 



\end{abstract}

\setcounter{section}{0}

\section{Introduction}

The main goal of this paper is to prove interior Lipschitz regularity for continuous viscosity solutions of fully nonlinear, conformally invariant, degenerate elliptic equations arising from conformal geometry.

Let $\mathbb{R}^{n}$ denote the Euclidean space of dimension $n$,
\begin{equation}
\Gamma\subset \mathbb{R}^{n}\mbox{ be an open convex symmetric cone with vertex at the origin},\label{cone1}
\end{equation}
satisfying
\begin{equation}
\Gamma \supset \{\lambda\in\mathbb{R}^{n}:\lambda_{i}>0,i=1,\cdots,n\}.
\label{wb}
\end{equation}
Let $f\in C^{1}(\Gamma)\cap C^{0}(\overline{\Gamma})$ be a symmetric function satisfying 
\begin{equation}
 f = 0 \text{ on } \partial\Gamma \text{ and } f>0,\quad \frac{\partial f}{\partial \lambda_{i}}>0 \mbox{ in }\Gamma \text{ for } i = 1, \ldots, n.\label{fGonzalez06}
\end{equation}
In the above, the symmetricity of $(f,\Gamma)$ is understood in the sense that if $\lambda \in \Gamma$, then $\tilde \lambda \in \Gamma$ and $f(\tilde\lambda) = f(\lambda)$ for any permutation $\tilde\lambda$ of $\lambda$.

For a function $u$ defined on a Euclidean domain, let $A^u$ denote its conformal Hessian matrix, i.e.
\[
A^u:=-\frac{2}{n-2}u^{-\frac{n+2}{n-2}}\nabla^2u+\frac{2n}{(n-2)^{2}}u^{-\frac{2n}{n-2}} \nabla u \otimes \nabla u -\frac{2}{(n-2)^{2}}u^{-\frac{2n}{n-2}}|\nabla u|^{2}I,
\]
where here and below $I$ denotes the $n\times n$ identity matrix, and, for $p, q \in \RR^n$, $p \otimes q$ denotes the $n\times n$ matrix with entries $(p \otimes q)_{ij} = p_i\,q_j$. Let $\lambda(A^u)$ denote the eigenvalues of $A^u$.

In recent years, there has been a growing literature on the following two equations:
\begin{equation}
f\left(\lambda(A^{u})\right)=1,\quad u>0 \mbox{ and } \quad \lambda(A^u) \in \Gamma,\label{yamabe'}
\end{equation}
and
\begin{equation}
\lambda(A^{u})\in\partial\Gamma,\quad \mbox{ and } \quad u>0.\label{main equation}
\end{equation}
Note that equation \eqref{main equation} is equivalent to 
\begin{equation*}
f\left(\lambda(A^{u})\right)=0,\quad  u>0\quad\mbox{and} \quad\lambda(A^{u})\in\overline{\Gamma}.
\end{equation*}
Equation \eqref{yamabe'} and \eqref{main equation} are second order fully nonlinear elliptic and degenerate elliptic equations, respectively. Fully nonlinear elliptic equations involving $f(\lambda(\nabla^{2}u))$ was investigated in the classic paper \cite{C-N-S-Acta}.

The equations \eqref{yamabe'} and \eqref{main equation} arose from conformal geometry. On a Riemannian manifold $(M,g)$ of dimension $n\geq3$, consider the Schouten tensor 
\begin{equation*}
A_{g}=\frac{1}{n-2}(\mbox{Ric}_{g}-\frac{1}{2(n-1)}R_{g}g),
\end{equation*}
where $\mbox{Ric}_{g}$ and $R_{g}$ denote, respectively, the Ricci tensor and the scalar curvature. Let $\lambda(A_{g})=(\lambda_{1},\cdots,\lambda_{n})$ denote the eigenvalues of $A_{g}$ with respect to $g$. It is well known that, in a conformal change of the metric, the ``main contribution'' to the curvature tensor is captured in the change of the Schouten tensor.  One is thus naturally led to study, in the hope of finding some sort of ``best metric'' in a conformal class of metrics, the problem (see e.g. \cite{CGY02-AnnM,Viac00-Duke})
\begin{equation}
f\left(\lambda(A_{u^{\frac{4}{n-2}}g})\right)=1,\quad u>0,\quad\mbox{and}\quad \lambda(A_{u^{\frac{4}{n-2}}g})\in\Gamma \mbox{ on }M.\label{yamabe}
\end{equation}
This problem is sometimes referred to in the literature as a fully nonlinear version of the Yamabe problem. When $M$ is a Euclidean domain and $g=g_{\rm flat}$ is the flat metric, equation \eqref{yamabe} is exactly equation \eqref{yamabe'}. Furthermore, both equation \eqref{yamabe'} and equation \eqref{main equation} appear naturally in the study of blow-up sequences of solutions of \eqref{yamabe} on manifolds.

Important examples of $(f,\Gamma)$ are $(f,\Gamma)=(\sigma_{k}^{\frac{1}{k}},\Gamma_{k})$, $1 \leq k \leq n$, where $\sigma_{k}(\lambda):=\sum\limits_{1\leq i_{1}<\cdots<i_{k}\leq n}\lambda_{i_{1}}\cdots\lambda_{i_{k}}$ is the $k$-th elementary symmetric function, and $\Gamma_{k}$ is the connected component of $\{\lambda\in\mathbb{R}^{n}:\sigma_{k}(\lambda)>0\}$ containing the positive cone $\{\lambda\in\mathbb{R}^{n}:\lambda_{i}>0,i=1,\cdots,n\}$. When $(f,\Gamma)=(\sigma_{1},\Gamma_{1})$, \eqref{yamabe} is the classical Yamabe problem in the so-called positive case.

In this paper, we establish the following regularity result for continuous viscosity solutions of \eqref{main equation}. See \cite[Definition 1.1]{Li09-CPAM} and Definition \ref{Def:ViscositySolution} below for the definition of viscosity solutions.
\begin{thm}[Local Lipschitz regularity] For $n\geq 3$, let $\Omega$ be an open subset of $\mathbb{R}^{n}$, and $\Gamma$ satisfy \eqref{cone1} and \eqref{wb}. Assume that $u$ is a continuous viscosity solution of \eqref{main equation} in $\Omega$. Then $u\in C^{0,1}_{loc}(\Omega)$.
\label{thm:regularity}
\end{thm}

\begin{rem}
As a consequence of Theorem \ref{thm:regularity},
several previously known results for  Lipschitz continuous solutions 
of \eqref{main equation}  hold for continuous solutions.
This includes
the Liouville-type Theorem 1.4, the symmetry results
Theorem 1.18 and Theorem 1.23 in \cite{Li09-CPAM};  the B\^ocher-type Theorems 1.2 and 1.3,
the Harnack-type Theorem 1.5, and
the asymptotic behavior results
 Corollary 1.7 and  Theorem 1.8 in \cite{LiNgBocher}.
\end{rem}

Although there have been many works on a priori estimates for solutions to \eqref{yamabe'} and \eqref{main equation} and closely related issues (see e.g. \cite{CGY02-AnnM, Chen05,GeWang06, Gonzalez05, GW03-IMRN,GV07,HLT10, J,J-L-L, LiLi03,LiLi05,Li09-CPAM,LiNgBocher,LiNgSymmetry,NadirashviliVladuts, STW07, TW09, Viac02-CAG, Wang06}), our theorem above appears to be the first regularity result for viscosity solutions in this context.

The regularity obtained in Theorem \ref{thm:regularity} is in a sense sharp: In \cite{NadirashviliVladuts}, Nadirashvili and Vl\u{a}du\c{t} showed that, for any $\epsilon \in (0,1)$, there exists a solution to a uniformly elliptic and conformally invariant equation in a ball $B \subset \RR^5$ which belongs to $C^{1,\epsilon}(B) \setminus C^{1,\epsilon+}(B)$.

It is sometimes more convenient to write $u = e^{-\frac{n-2}{2}\psi}$ or $u = w^{-\frac{n-2}{2}}$. An easy computation gives $A^u = A_w = e^{2\psi} A[\psi]$ where
\begin{align*}
A_{w}
	&=w\nabla^{2}w-\frac{1}{2}|\nabla w|^{2}I,\\
A[\psi] 
	&=  \nabla^2 \psi 
	+  \nabla\psi \otimes \nabla\psi
	- \frac{1}{2} |\nabla \psi|^{2}I.
\end{align*}

In addition to Theorem \ref{thm:regularity}, we also study the Dirichlet boundary value problem for a class of degenerate elliptic equations which includes the conformal operator $A[\psi]$. Consider operators of the form
\begin{equation}
F[\psi] = \nabla^2 \psi + \alpha\,\nabla \psi \otimes \nabla \psi - \beta |\nabla \psi|^2\,I
	\label{Eq:FConstab}
\end{equation}
where $\alpha$ and $\beta$ are constant, and the equation 
\[
F[\psi] \in \partial U,
\]
where $U$ is a non-empty open subset of $\Snn$
satisfying a degenerate ellipticity condition:
\begin{equation}
\text{ if }A\in U, B\in\Snn\mbox{ and }B>0, \text{ then } A+B\in U.
	\label{Eq:UCondPos}
\end{equation}
(Note that \eqref{Eq:UCondPos} implies that $\partial U$ is Lipschitz.)

In the context of Theorem \ref{thm:regularity}, $U$ is the set of symmetric matrices whose eigenvalues belong to $\Gamma$, as equation \eqref{main equation} can be written equivalently as $A^u \in \partial U$. We note for future use that our results below apply also to the setting of fully nonlinear Yamabe problem of ``negative type'' by considering the set $U$ of symmetric matrices whose eigenvalues belong to $\RR^n \setminus (-\bar\Gamma)$, where
\[
-\bar\Gamma = \{\lambda \in \RR^n: -\lambda \in \bar\Gamma\}.
\]
In both cases, \eqref{Eq:UCondPos} holds thanks to \eqref{cone1} and \eqref{wb}.

For any set $S \subset\mathbb{R}^{n}$, we use $\mbox{USC}(S)$ to denote the set of functions $\psi:S\rightarrow\mathbb{R}\cup\{-\infty\}$, $\psi \not\equiv -\infty$ in $S$, satisfying 
\begin{equation*}
\limsup\limits_{x\rightarrow\bar{x}}\psi(x)\leq \psi(\bar{x}),\quad \forall \bar{x}\in S.
\end{equation*}
Similarly, we use $\mbox{LSC}(S)$ to denote the set of functions $\psi: S\rightarrow\mathbb{R}\cup\{+\infty\}$, $\psi \not\equiv +\infty$  in $S$, satisfying 
\begin{equation*}
\liminf\limits_{x\rightarrow\bar{x}}\psi(x)\geq \psi(\bar{x}),\quad \forall \bar{x}\in S.
\end{equation*}

We now give the definition of viscosity subsolutions, supersolutions and solutions to the degenerate elliptic equation $F[\psi] \in \partial U$.

\begin{Def}\label{Def:ViscositySolution}
Let $\Omega\subset\mathbb{R}^{n}$, $n \geq 2$, be an open set, and $U$ be a non-empty open 
subset of $\Snn$ satisfying \eqref{Eq:UCondPos}. For a function  $\psi$ in $USC(\Omega)$ ($LSC(\Omega)$), we say that 
\begin{equation*}
F[\psi] \in \overline{U}\quad \left(F[\psi]\in\Snn\setminus U\right) \quad\mbox{in }\Omega \quad\mbox{in the viscosity sense}
\end{equation*}
if for any $x_{0}\in\Omega$, $\varphi\in C^{2}(\Omega)$, $(\psi-\varphi)(x_{0})=0$ and 
\begin{equation*}
\psi-\varphi\leq0\quad(\psi-\varphi\geq0),\quad\mbox{near }x_{0},
\end{equation*}
there holds
\begin{equation*}
F[\varphi](x_{0})\in \overline{U}\quad \left(F[\varphi](x_{0})\in\Snn\setminus U\right).
\end{equation*}

We say that a function $\psi \in C^0(\Omega)$ satisfies 
\begin{equation}
F[\psi]\in \partial U \text{ in the viscosity sense}
	\label{Eq:FpsiEq}
\end{equation}
in $\Omega$ if $F[\psi]$ belongs to both $\overline{U}$ and $\Snn\setminus U$ in $\Omega$ in the viscosity sense.
 
 When $F[\psi] \in \overline{U}\quad \left(F[\psi]\in\Snn\setminus U\right)$ in $\Omega$ in the viscosity sense, we also say interchangeably that $\psi$ is a viscosity subsolution (supersolution) to \eqref{Eq:FpsiEq} in $\Omega$.
\end{Def}

Our next result is a uniqueness statement for \eqref{Eq:FpsiEq} when $U$ satisfies
\begin{equation}
A \in U \text{ and } c > 0 \Rightarrow cA \in U.
	\label{Eq:UCone}
\end{equation}

\begin{thm}[Uniqueness for the Dirichlet Problem]
Let $\Omega\subset\mathbb{R}^{n}$ ($n \geq 2$) be a non-empty bounded open set, and $U$ be a non-empty open subset of $\Snn$ satisfying \eqref{Eq:UCondPos} and \eqref{Eq:UCone}. Assume that $F$ is of the form \eqref{Eq:FConstab}. Then, for any $\varphi \in C^0(\partial\Omega)$, there exists at most one solution $\psi \in C^{0}(\bar\Omega)$ of \eqref{Eq:FpsiEq} satisfying $\psi = \varphi$ on $\partial\Omega$.
\label{thm:Uniq}
\end{thm} 

We also prove the following existence theorem using Perron's method (see \cite{Ishii89-CPAM}). 

\begin{thm}[Existence by sub- and supersolution method]\label{thm:Perron}
Let $\Omega$ and $(F,U)$ be as in Theorem \ref{thm:Uniq}. Let
$w\in LSC(\overline \Omega)$ and $v\in USC(\overline\Omega )$ be respectively supersolution and subsolution of \eqref{Eq:FpsiEq} in $\Omega$ such that $w\geq v$ in $\Omega$ and $w=v$ on $\partial\Omega$. Then there exists a viscosity solution $u\in C^0(\overline \Omega)$ of \eqref{Eq:FpsiEq} in $\Omega$ satisfying
\begin{align*}
v\le u\le w &\qquad\mbox{in}\ \overline\Omega,\\
u=w=v &\qquad \mbox{on}\ \partial\Omega.
\end{align*}
\end{thm}

One main ingredient of the proofs of Theorems \ref{thm:regularity}, \ref{thm:Uniq} and \ref{thm:Perron} is a comparison principle. In recent years, comparison principles (for viscosity solutions) have been very successfully applied to derive estimates and symmetry properties for solutions to (both degenerate and non-degenerate elliptic) equations in conformal geometry; see \cite{Li09-CPAM} and the references therein. Our paper can be viewed as a continuation in this line of work.

In fact, we will establish a variant of the comparison principle for more general operators of the form
\begin{equation}
F[\psi] = \nabla^2 \psi + \alpha(\cdot, \psi)\,\nabla \psi \otimes \nabla \psi - \beta(\cdot, \psi) |\nabla \psi|^2\,I
	\label{Eq:FIntroDef}
\end{equation}
where $\alpha, \beta: \Omega \times \RR \rightarrow \RR$ and $\Omega$ is an open subset of $\RR^n$. Throughout the paper, we will assume that
\begin{equation}
\text{the function $L(x,s,p) := \alpha(x,s) p \otimes p - \beta(x,s)|p|^2 I$ is non-decreasing in $s$.}
	\label{Eq:LMonotone}
\end{equation}
Note that this condition is consistent with both $A[\psi]$ and $\frac{1}{w}A_w$ defined above.

In the sequel, we say that \emph{the principle of propagation of touching points} holds for $(F,U)$ if for any supersolution $w \in LSC(\bar\Omega)$ and subsolution $v\in USC(\bar\Omega)$ of \eqref{Eq:FpsiEq} in $\Omega$ one has
\[
w\geq v \mbox{ in } \Omega \text{ and }  w > v \text{ on } \partial\Omega \qquad\Rightarrow \qquad w > v \text{ in } \Omega.
\]
(In other words, if $w \geq v$ in $\Omega$ then every non-empty connected component of the set $\{x \in \bar\Omega: w(x) = v(x)\}$ contains a point of $\partial\Omega$.) This principle can be viewed as a weak version of the strong comparison principle.

We say that \emph{the comparison principle} holds for $(F,U)$ if for any supersolution $w \in LSC(\bar\Omega)$ and subsolution $v\in USC(\bar\Omega)$ of \eqref{Eq:FpsiEq} in $\Omega$ one has
\[
w \geq v \text{ on } \partial\Omega \qquad \Rightarrow \qquad w \geq v \text{ in }  \Omega.
\]
It should be noted that, for general degenerate elliptic equations, $w \geq v$ in $\Omega$ does not imply the dichotomy that $w > v$ or $w \equiv v$ in $\Omega$. (This is in contrast with the uniformly elliptic case.)

\begin{rem}\label{rem:SCP=>CP}
If $L(x,s,p)$ is independent of $s$, then the principle of propagation of touching points is equivalent to the comparison principle. 
\end{rem}

We prove that the principle of propagation of touching points holds when $(F,U)$ satisfies, in addition to \eqref{Eq:UCondPos}, \eqref{Eq:UCone} and \eqref{Eq:LMonotone}, the following structural conditions:
\begin{equation}
\begin{array}{ll}
\text{ either }& \text{$|\beta(x, s)| > \beta_0 > 0$ for some constant $\beta_0$}\\
\text{ or } & \text{ both $\alpha$ and $\beta$ are constant}.
\end{array}
	\label{Eq:betaStruct}
\end{equation}

\begin{thm}[Principle of propagation of touching points]\label{thm:CPQuad}
Let $F$ be of the form \eqref{Eq:FIntroDef} where $\alpha, \beta \in C^{0,1}_{loc}(\bar\Omega \times \RR)$ satisfy \eqref{Eq:LMonotone} and \eqref{Eq:betaStruct}. Let $\Omega\subset\mathbb{R}^{n}$ ($n \geq 2$) be a non-empty bounded open set, and $U$ be a non-empty open subset of $\Snn$ satisfying \eqref{Eq:UCondPos} and \eqref{Eq:UCone}. Assume that $w \in LSC(\bar\Omega)$ and $v\in USC(\bar\Omega)$ are respectively a supersolution and a subsolution of \eqref{Eq:FpsiEq} in $\Omega$.

\begin{enumerate}[(a)]
\item If $w \geq v$ in $\Omega$ and $w > v$ on $\partial\Omega$, then $w > v$ in $\Omega$.

\item In case $\alpha$ and $\beta$ are constant, if $w \geq v$ on $\partial\Omega$, then $w \geq v$ in $\Omega$.
\end{enumerate}
\end{thm}

When $w$ and $v$ are locally Lipschitz and $F[\psi] = A[\psi]$, Theorem \ref{thm:CPQuad} was established in \cite{Li09-CPAM}. 

The proof of Theorem \ref{thm:CPQuad} yields the propagation principle for an even larger class of operators; see Theorem \ref{thm:CPNUE} (where the assumption on the quadratic dependence of $F[\psi]$ on $\nabla\psi$ is somewhat relaxed to a super-linear dependence). One ingredient of the proof is a first variation result which, roughly speaking, allows one to perturb a given function $\psi$ to another function $\tilde\psi$ such that $F[\tilde\psi]$ are either ``more inside'' or ``more outside'' the set $U$ than $F[\psi]$ in a detailed controlled fashion. There are two delicate points of this process. On the one hand, one needs to ensure that the gain obtained is strong enough to counter-balance the error accrued in either regularization or handling the difficulties created by degenerate ellipticity. On the other hand, the whole process is carried out in such a way that it depends only on an upper bound and a lower bound of $\psi$, and not on $\nabla \psi$ or $\nabla^2 \psi$. It is in this first variation argument that the assumptions that $\beta$ does not change sign and $L$ is non-decreasing are crucially used. See subsection \ref{Sec:CPCounterex} for examples which hint that these assumptions cannot simply be dropped.

Comparison principles for different classes of (degenerate) elliptic operators are available in the literature. See \cite{AmendolaGaliseVitolo13-DIE, BardiDaLio99-AM, HBDolcettaPorretaRosi15-JMPA, BirindelliDemengel04-AFSTM, BirindelliDemengel07-CPAA, UserGuide, DolcettaVitolo07-MC, DolcettaVitolo-preprint, HartmanNirenberg, HarveyLawsonSurvey2013, Ishii89-CPAM, IshiiLions90-JDE,Jensen88-ARMA, KawohlKutev98-AM, KawohlKutev00-FE, KawohlKutev07-CPDE, KoikeKosugi15-CPAA, KoikeLey11-JMAA, Trudinger88-RMI} and the references therein. Most of these works assumed a kind of ``properness/non-degeneracy'' of the operator with respect to the unknown $\psi$ (e.g. $L$ is decreasing with respect to $s$) which is not applicable to our setting (see condition \eqref{Eq:LMonotone}). In the present paper, we exploit instead some non-degeneracy with respect to the derivatives $\nabla \psi$ of the unknown.

It is natural to ask if the method we follow here can be tweaked together with the more familiar treatment for proper operators to treat a broader class of operators, but this goes beyond the scope of the present paper. We however note that (non-strict) properness of the operator is far from ensuring the validity of a comparison/propagation principle; see Proposition \ref{prop:CtexNonDec}.

Using results on removable singularities in \cite{CafLiNir11}, we obtain the following comparison principle on domains with singularities when $U$ satisfies in addition the condition
 \begin{equation}
U\subset\{M\in\Snn:\mbox{tr}(M)>0\},\label{gooo}
 \end{equation}
where $\mbox{tr}(M)$ denotes the trace of $M$. For the proof, see Section \ref{sec:CP}.

\begin{cor}
\label{cor:CPSing}
Let $\Omega \subset \RR^n$ ($n \geq 2$) be a bounded non-empty open set, $E\subset \Omega$ be a closed set with zero Newtonian capacity, and $(F,U)$ be as in Theorem \ref{thm:CPQuad} with constant $\alpha$ and $\beta$. If $w\in LSC(\bar\Omega)$ and $v\in USC(\bar\Omega \setminus E)$ satisfy
\begin{align*}
w &\text{ is a supersolution to \eqref{Eq:FpsiEq} in $\Omega$},\\
v &\text{ is a subsolution to \eqref{Eq:FpsiEq} in $\Omega \setminus E$},
\end{align*} 
$w \geq v$ in $\Omega \setminus E$ and $w > v$ on $\partial\Omega$, and if
\begin{align}
\text{either } & \sup_{\Omega \setminus E} v < + \infty,\label{Eq:anb1}\\
\text{or } & \alpha - n \beta < 0,\label{Eq:anb2}
\end{align}
 then $\inf_{\Omega \setminus E} (w - v) > 0$.
\end{cor}

\begin{rem}
It is interesting to identify the set $S_k$ of $(\alpha,\beta)$ for which one cannot drop the assumption that $v$ is bounded from above when $U$ is the set of symmetric matrices whose eigenvalues belong to $\Gamma_k$ with $2 \leq k \leq n$. Note that by the above result, $S_k \subset \{\alpha - n\beta \geq 0\}$. For $k = 1$, equation \eqref{Eq:FpsiEq} becomes $\Delta\psi + (\alpha - n\beta)|\nabla \psi|^2 = 0$, from which one can see that $S_1 \subset \{\alpha - n\beta > 0\}$. In fact, $S_1 = \{\alpha - n\beta > 0\}$. To see this, note that the functions $\psi_\mu(x)=\frac{1}{\alpha - n\beta} \ln(|x|^{2-n} + \mu)$ with $\mu \geq 0$ are solutions of \eqref{Eq:FpsiEq} in $B_1(0) \setminus \{0\}$. In particular, $w = \psi_1$ is a supersolution of \eqref{Eq:FpsiEq} in $B_1(0)$, $v = \psi_0$ is a subsolution of \eqref{Eq:FpsiEq} in $B_1(0) \setminus \{0\}$, $w \geq v$ in $B_1(0) \setminus \{0\}$, and $w > v$ on $\partial B_1(0)$, but $\inf_{B_1(0) \setminus \{0\}} (w - v) = 0$.
\end{rem}

When $F[\psi]$ is the conformal operator $A[\psi]$, Corollary \ref{cor:CPSing} was proved by the first named author in \cite{Li09-CPAM} under the assumption that $E\subset \Omega$ containing at most finitely many points, $U=\{M\in\Snn:\lambda(M)\in\Gamma\}$, $w\in C^{0,1}(\bar\Omega)$ and $v\in C^{0,1}(\bar\Omega\setminus E)$.

A related issue of interest is whether the strong maximum principle and the Hopf lemma holds. It turns out that, in this degenerate elliptic setting, both fail for a large class of operators. See \cite{CafLiNir11, LiNir-misc} for further discussion.

Last but not least, the proof of Theorem \ref{thm:regularity} uses not only the comparison principle (Theorem \ref{thm:CPQuad}) but also conformal invariance properties of $A[\psi]$ (i.e. of $A^u$). We remark that, for general $F$ and $U$, the comparison principle itself is far from ensuring (Lipschitz) regularity of viscosity solutions. See Section \ref{Sec:LipRegVS} for further discussion.

The rest of the paper is structured as follows. We start in Section \ref{Sec:Prelim} with some preliminaries about regularizations of semi-continuous functions by lower and upper envelops. In Section \ref{sec:CP}, we prove a generalization of Theorem \ref{thm:CPQuad} for more general operators and give counterexamples to highlight the importance of the conditions in the theorem. In Section \ref{Sec:Perron}, we prove the uniqueness Theorem \ref{thm:Uniq} and the existence result Theorem \ref{thm:Perron}. Finally, in Section \ref{Sec:LipRegVS}, we prove the regularity result Theorem \ref{thm:regularity} together with some generalization.

%+++++++++++++++++%
\section{Preliminaries}\label{Sec:Prelim}

We briefly recall a well-known regularization of semi-continuous functions which will be used later in the paper.

Assume $n\geq1$ and let $\Omega$ be an open bounded set in $\mathbb{R}^{n}$. For a function $v \in USC(\bar \Omega)$ and $\epsilon > 0$, we define the $\epsilon$-upper envelop of $v$ by
\begin{equation}
v^{\epsilon}(x)
	:=\max\limits_{y\in\bar\Omega}\Big\{v(y) -\frac{1}{\epsilon}|y-x|^{2}\Big\}, \qquad \forall x \in \bar \Omega.
		\label{Eq:UpEnvDef}
\end{equation}
Likewise, for a function $w \in LSC(\bar \Omega)$, its $\epsilon$-lower envelop is defined by
\begin{equation}
w_{\epsilon}(x):=\min\limits_{y\in\bar\Omega}\Big\{w(y) +\frac{1}{\epsilon}|y-x|^{2}\Big\},\qquad\forall x\in\bar\Omega.
	\label{Eq:LowEnvDef}
\end{equation}

Although our definition of upper and lower envelops is slightly different from the definition in \cite{CabreCaffBook}, all relevant properties established in \cite[Lemma 5.2]{CabreCaffBook} remain valid with minor modification. We collect below some useful properties.

\begin{enumerate}[(i)]
\item\label{UpLowPropi} $v^\epsilon, w_\epsilon$ belong to $C(\bar\Omega)$, are monotonic in $\epsilon$ and 
\begin{equation}
\text{$v_\epsilon \rightarrow v$, $w_\epsilon \rightarrow w$ pointwise as $\epsilon \rightarrow 0$.}
	\label{Eq:UpLowConv}
\end{equation}
\item\label{UpLowPropii} $v^{\epsilon}$ and $w_{\epsilon}$ are punctually second order differentiable (see e.g. \cite{CabreCaffBook} for a definition) almost everywhere in $\Omega$ and 
\begin{equation}
\nabla^{2}v^{\epsilon}\geq-\frac{2}{\epsilon}I,\quad\nabla^{2}w_{\epsilon}\leq \frac{2}{\epsilon}I,\quad\mbox{ a.e. in }\Omega.\label{utr}
\end{equation}
\item\label{UpLowPropiii} For any $x \in \Omega$, there exists $x^* = x^*(x) \in \bar \Omega$ such that 
$$v^\epsilon(x) = v(x^*)  - \frac{1}{\epsilon}|x^* - x|^2 \text{ and } |x^* - x|^2 \leq \epsilon(\max_{\bar\Omega} v - v(x)).$$
Likewise, for any $x \in \Omega$, there exists $x_* = x_*(x) \in \bar \Omega$ such that 
$$w_\epsilon(x) = w(x_*)  + \frac{1}{\epsilon}|x_* - x|^2 \text{ and } |x_* - x|^2 \leq \epsilon(w(x) - \min_{\bar\Omega} w).$$

\item \label{UpLowPropiv} If it holds for some non-empty open subset $\omega$ of $\Omega$ that $\inf_{\omega} v > -\infty$ and $\sup_{\omega} w < +\infty$, then
\begin{equation}
|\nabla v^\epsilon| \leq \frac{2}{\epsilon^{\frac{1}{2}}} \big[\max_{\bar\Omega} v - \inf_{\omega} v\big]^{\frac{1}{2}} \text{ and } |\nabla w_\epsilon| \leq \frac{2}{\epsilon^{\frac{1}{2}}} \big[\sup_{\omega} w - \min_{\bar\Omega} w\big]^{\frac{1}{2}}
	\label{Eq:UpLowGradEst}
\end{equation}
almost everywhere in $\omega$.

\item \label{UpLowPropv} The bounds for $|x^* - x|$ and $|x_* - x|$ in \eqref{UpLowPropiii} can be improved when $v$ and $w$ are more regular. In fact, if $|v(x) - v(y)| \leq m(|x - y|)$ for all $x, y \in \bar\Omega$ and for some non-negative continuous non-decreasing function $m: [0,\infty) \rightarrow [0,\infty)$ satisfying $m(0) = 0$, then
\[
|x^* - x| \leq \big[\epsilon\,m((2\epsilon\,\sup_{\bar\Omega} |v|)^{1/2})\big]^{1/2}.
\]
Analogously, if $|v(x) - v(y)| \leq m(|x - y|)$ for all $x, y \in \bar\Omega$, then
\[
|x_* - x| \leq \big[\epsilon\,m((2\epsilon\,\sup_{\bar\Omega} |w|)^{1/2})\big]^{1/2}.
\]
Nevertheless, the bounds for $|x^* - x|$ and $|x_* - x|$ in \eqref{UpLowPropiii} are generally sharp for semi-continuous functions.
\end{enumerate}

Properties \eqref{UpLowPropi}-\eqref{UpLowPropiii} can be found in \cite{CabreCaffBook}. To see Property \eqref{UpLowPropiv}, we let $x_0 \in \omega$ be a point of differentiablity of $v^\epsilon$, and estimate, for $x_1 \in \Omega$,
\begin{align*}
v^\epsilon(x_0) 
	&\geq v((x_1)^*)  - \frac{1}{\epsilon}|(x_1)^* - x_0|^2\\
	&\geq v((x_1)^*)  - \frac{1}{\epsilon}|(x_1)^* - x_1|^2  - \frac{2}{\epsilon}|(x_1)^* - x_1||x_1 - x_0| -  \frac{1}{\epsilon}|x_1 - x_0|^2\\
	&= v^\epsilon(x_1) - \frac{2}{\epsilon}|(x_1)^* - x_1||x_1 - x_0| -  \frac{1}{\epsilon}|x_1 - x_0|^2, \text{ for all } x_0, x_1 \in \Omega,
\end{align*}
which implies, in view of Property \eqref{UpLowPropiii}, that
\[
\frac{v^\epsilon(x_1) - v^\epsilon(x_0)}{|x_1 - x_0|} \leq \frac{2}{\epsilon^{\frac{1}{2}}}|\max_{\bar\Omega} v - v(x_1)| + \frac{1}{\epsilon}|x_1 - x_0|.
\]
Sending $x_1 \rightarrow x_0$ and recalling Property \eqref{UpLowPropii}, we obtain the assertion. Property \eqref{UpLowPropv} follows from \eqref{UpLowPropiii} and the estimate
\begin{align*}
\frac{1}{\epsilon}|x^* - x|^2 
	&= v(x^*)  - v^\epsilon(x) \leq v(x^*) - v(x) \leq m(|x^* - x|).
\end{align*}

The sharpness of the estimates for $|x^* - x|$ and $|x_* - x|$ in \eqref{UpLowPropiii} is demonstrated by the following example. Consider $\Omega = (-1,1)$. For $x \in[-1,1]$, define
\[
w(x) = \left\{\begin{array}{ll}
	1 &\text{ if } 2^{-(k+2)} < |x| \leq 2^{-(k+1)} \text{ for some } k \geq 0,\\
	2 - 2^{k+1}|x| &\text{ if } 2^{-(k+1)} < |x| \leq 2^{-k} \text{ for some } k \geq 0,\\
	0 &\text{ if } x = 0.
\end{array}\right.
\]
Then $w \in LSC([-1,1]) \cap L^\infty(-1,1)$. For $k > 1$, let $\epsilon_k = 2^{-2(2k+1)}$ and $x_k = 2^{-(2k+3)}$. We have
\[
w_{\epsilon_k}(x_k) \leq w(0)  + \frac{1}{\epsilon_k}|x_k|^2 = \frac{1}{16}.
\]
On the other hand, for $|y - x_k| < 2^{-(2k+4)} = \frac{1}{8} \sqrt{\epsilon_k}$, we have $w(y) > \frac{1}{2}$ and 
\[
w(y)  + \frac{1}{\epsilon_k}|y - x_k|^2 \geq \frac{1}{2} .
\]
It follows that $|(x_k)_* - x_k| \geq \frac{1}{8} \sqrt{\epsilon_k}$.

We conclude the section with a simple lemma about the stability of envelops with respect to semi-continuity.

\begin{lem}\label{lem:EquiSC}
Assume that $v \in USC(\bar\Omega)$ and $\inf_{\bar\Omega} v > -\infty$. Then for all sequences $\epsilon_j \rightarrow 0$ and $x_j \rightarrow x \in \Omega$, there holds
\[
\limsup_{j \rightarrow \infty} v^{\epsilon_j}(x_j) \leq v(x).
\]
Likewise, if $w \in LSC(\bar\Omega)$ and $\sup_{\bar\Omega} w <  +\infty$, then
\[
\liminf_{j \rightarrow \infty} w_{\epsilon_j}(x_j) \geq w(x).
\]
\end{lem}

\begin{proof} We will only show the first assertion. Assume by contradiction that there exist some sequences $\epsilon_j \rightarrow 0$, $x_j \rightarrow x \in \Omega$ such that
\[
v^{\epsilon_j}(x_j) \geq v(x) + 2\delta \text{ for some } \delta > 0.
\]
By the semi-continuity of $v$, there exists $\theta > 0$ such that 
\[
v(y) \leq v(x) + \delta \text{ for all } |y - x| < \theta.
\]
By property \eqref{UpLowPropiii}, there exists $\hat x_j$ such that
\[
v^{\epsilon_j}(x_j) = v(\hat x_j) - \frac{1}{\epsilon_j}|x_j - \hat x_j|^2 \text{ and } |x_j - \hat x_j|^2 \leq \epsilon_j (\sup_{\bar\Omega} v - v(x_j)) \rightarrow 0,
\]
where we have used $\inf_{\bar\Omega} v > -\infty$. It then follows that $|\hat x_j - x| < \theta$ for all sufficiently large $j$ and so
\[
v^{\epsilon_j}(x_j) \leq v(\hat x_j) \leq v(x) + \delta,
\]
which amounts to a contradiction.
\end{proof}

\section{The principle of propagation of touching points}\label{sec:CP}

In this section, we prove Theorem \ref{thm:CPQuad}. We will establish the propagation principle for more general operators of the form
\begin{equation}
F[\psi] = \nabla^2 \psi + L(\cdot,\psi,\nabla\psi),
 \label{Eq:FDef}
\end{equation}
where $L: \Omega \times \RR \times \RR^n \rightarrow \mathcal{S}^{n \times n}$, under some structural assumptions on $L$ and $U$ which we will detail below. (Clearly, Definition \ref{Def:ViscositySolution} extends to this general setting.)

It is natural to require that $L$ be locally Lipschitz continuous. When $L$ is only H\"older continuous, the propagation principle fails in a manner similar to the non-uniqueness of first order ODE with non-Lipschitz right hand side. The following example is well-known: Consider the equation $\Delta \psi = |\nabla \psi|^\gamma$ with $\gamma \in (0,1)$, i.e. $F[\psi] = \nabla^2 \psi - |\nabla \psi|^\gamma I$ and $U = \{M \in \Snn: tr(M) > 0\}$. This equation admits $\psi(x) \equiv 0$ and $\hat\psi(x) = \frac{1}{(\lambda + 1)(\lambda + n - 1)^{\lambda}} |x|^{\lambda+1}$ as classical solutions, where $\lambda  = \frac{1}{1 - \gamma}$. As $\hat \psi \geq \psi$ on $\RR^n$ and equality holds only at $x = 0$, the propagation principle fails.

We note that the degenerate ellipticity condition \eqref{Eq:UCondPos} and local Lipschitz regularity of $L$ are far from enough to ensure the correctness of the propagation principle (even for rotationally symmetric and \emph{proper} operators); see subsection \ref{Sec:CPCounterex} for counterexamples.

The following structural conditions on $(F,U)$ are directly motivated by the conformal operator $A[\psi]$. First, we assume that $U$ satisfies 
\begin{equation}
A\in U, c\in(0,1)\Rightarrow cA\in U.
\label{Eq:UCond*S}
\end{equation}
Second, we assume that, for every $R > 0$ and $\Lambda > 0$, there exist $m \geq 0$, $\bar \theta > 0$ and $C > 0$ such that, for $x \in \Omega$ and $p \in \RR^n$,
\begin{align}
|\nabla_x L(x,s,p)| 
	&\leq C|p|^m \qquad \forall~ |s| \leq R,
	\label{Eq:Lcond0}\\
0 \leq L(x,s', p) - L(x,s,p) 
	 &\leq C(s' - s)\,|p|^m\,I \qquad \forall~ -R \leq s \leq s' \leq R,
	 \label{Eq:Lcond1}
\end{align}
\begin{align}
&p \cdot \nabla_p L(x, s,p) - L(s,x,p )
\nonumber\\
	&\qquad 
	+ \theta\Lambda |\nabla_{p} L(x, s,p)|\,I  - \theta\,I
	 \leq C p \otimes p - \frac{1}{C}\, |p|^m\,I \qquad\forall~\theta \in [0,\bar\theta], |s| \leq R.
	 \label{Eq:Lcond2}
\end{align}
Note that, \eqref{Eq:Lcond1} and \eqref{Eq:Lcond2} should be understood as inequalities between real symmetric matrices: $M \leq N$ if and only if $N - M$ is non-negative definite. Also, \eqref{Eq:Lcond1} implies that $L$ is non-decreasing in $s$.

\begin{example}\label{Ex:QuadF} For all $m \geq 2$ and $\alpha, \beta \in C^{0,1}_{loc}(\RR)$ such that $\beta(s) > \beta_0 > 0$ for some constant $\beta_0$, $\alpha$ is non-decreasing and $\beta$ is non-increasing, the operator
\[
F[\psi] = \nabla^2 \psi + \alpha(\psi)\,\nabla\psi \otimes \nabla\psi - \beta(\psi)\,|\nabla \psi|^m\,I
\]
satisfies conditions \eqref{Eq:Lcond0}-\eqref{Eq:Lcond2}.
\end{example}

We now state our principle of propagation of touching points for operators of the form \eqref{Eq:FDef}.

\begin{thm}
Let $\Omega\subset\mathbb{R}^{n}$ ($n \geq 2$) be a non-empty bounded open set, $L: \Omega \times \RR \times \RR^n \rightarrow \Snn$ be locally Lipschitz continuous and satisfy \eqref{Eq:Lcond0}, \eqref{Eq:Lcond1} and \eqref{Eq:Lcond2} for some $m > 1$, $F$ be given by \eqref{Eq:FDef} and $U$ be a non-empty open 
subset of $\Snn$ satisfying \eqref{Eq:UCondPos} and \eqref{Eq:UCond*S}. If $w \in LSC(\bar\Omega)$ and $v\in USC(\bar\Omega)$ are respectively a supersolution and a subsolution of \eqref{Eq:FpsiEq} in $\Omega$, and if $w \geq v$ in $\Omega$ and $w > v$ on $\partial\Omega$, then $w > v$ in $\Omega$.
\label{thm:CPNUE}
\end{thm}

Interchanging the role of $\psi$ and $-\psi$ and of $U$ and $\mathcal{S}^{n \times n} \setminus (- \bar U)$ (where $-\bar U = \{-M: M \in \bar U\}$), we see that an analogous result holds if one replaces \eqref{Eq:UCond*S} by
\begin{equation}
A\in U, c\in(1,\infty) \Rightarrow cA\in U,
\label{Eq:UCond*SCat}
\end{equation}
and \eqref{Eq:Lcond2} by: for every $R > 0$ and $\Lambda > 0$, there exist positive constants $\bar \theta, C > 0$ such that, for $0 < \theta \leq \bar\theta$, $x \in \Omega$, $|s| \leq R$ and $p \in \RR^n$,
\begin{align}
&p \cdot \nabla_p L(x, s,p) - L(x, s, p ) 
\nonumber\\
	&\qquad\qquad - \theta\Lambda |\nabla_{p} L(x, s,p)|\,I  + \theta\,I
	 \geq -C p \otimes p + \frac{1}{C}\, |p|^m\,I .
	\label{Eq:Lcond2Cat}
\end{align}
We then obtain an equivalent statement of Theorem \ref{thm:CPNUE}:

\begin{thm}
Let $\Omega\subset\mathbb{R}^{n}$ ($n \geq 2$) be a non-empty bounded open set, $L: \Omega \times \RR \times \RR^n \rightarrow \Snn$ be locally Lipschitz continuous and satisfy \eqref{Eq:Lcond0}, \eqref{Eq:Lcond1} and \eqref{Eq:Lcond2Cat} for some $m > 1$, $F$ be given by \eqref{Eq:FDef} and $U$ be a non-empty open 
subset of $\Snn$ satisfying \eqref{Eq:UCondPos} and \eqref{Eq:UCond*SCat}. If $w \in LSC(\bar\Omega)$ and $v\in USC(\bar\Omega)$ are respectively a supersolution and a subsolution of \eqref{Eq:FpsiEq} in $\Omega$ and if $w \geq v$ in $\Omega$ and $w > v$ on $\partial\Omega$, then $w > v$ in $\Omega$.
\label{thm:CPNUECat}
\end{thm}

Assuming the correctness of the above theorem for the moment, we proceed with the 
\begin{proof}[Proof of Theorem \ref{thm:CPQuad}]
If $\beta > \beta_0 > 0$, the result is covered by Theorem \ref{thm:CPNUE}. If $\beta < -\beta_0 < 0$, the result is covered by Theorem \ref{thm:CPNUECat}. It remains to consider the case $\beta \equiv 0$ and $\alpha$ is constant. The operator $F$ then takes the form
\[
F[\psi] = \nabla^2 \psi + \alpha\,\nabla\psi \otimes \nabla \psi.
\]
When $\alpha \neq 0$, we note that the functions $\tilde w = \frac{\alpha}{|\alpha|} e^{\alpha w}$ and $\tilde v = \frac{\alpha}{|\alpha|} e^{\alpha v}$ satisfy $\tilde w \in LSC(\bar\Omega)$, $\tilde v\in USC(\bar\Omega)$ and, in view of \eqref{Eq:UCone},
\[
\nabla^2 \tilde w = |\alpha|\,|\tilde w|\,F[w] \in \Snn \setminus U \text{ and } \nabla^2 \tilde v = |\alpha|\,|\tilde v| F[v] \in \bar U.
\]
Therefore, we can assume without loss of generality that $\alpha = 0$, i.e.
\[
F[\psi] = \nabla^2 \psi.
\]

In this case, note that 
\begin{equation}
F[\psi + \mu\,|x|^2]  = F[\psi] + 2\mu\,I.
	\label{Eq:FVab=0}
\end{equation}
An easy adaption of the proof of Theorem \ref{thm:CPNUE} below (but using \eqref{Eq:FVab=0} instead of Lemma \ref{Lem:FVSub}) yields the result.
\end{proof}

We turn now to the proof of Theorem \ref{thm:CPNUE}.

%+++++++++++++++++%

\subsection{Error in regularizations}

The following result is a direct adaption of \cite[Theorem 5.1]{CabreCaffBook} which estimates the error to \eqref{Eq:FpsiEq} when making regularizations by lower and upper envelops.

\begin{prop}\label{key lemma}
Assume $n \geq 2$. Let $\Omega\subset\mathbb{R}^{n}$ be a bounded open set, $U$ be an open subset of $\Snn$ satisfying \eqref{Eq:UCondPos}, $L: \Omega \times \RR \times \RR^n \rightarrow \Snn$ be a locally Lipschitz continuous function satisfying \eqref{Eq:Lcond0} and the second inequality in \eqref{Eq:Lcond1} for some $m \geq 0$, and $F$ be given by \eqref{Eq:FDef}. For any $M > 0$, there exists $a > 0$ such that if $w \in LSC(\Omega)$ is a supersolution of \eqref{Eq:FpsiEq} in $\Omega$ and if $w_\epsilon$ is punctually second order differentiable at a point $x \in \Omega$ and $|w_\epsilon(x)| + |w(x_*)| \leq M$, then
\begin{align*}
&F[w_\epsilon](x) - a|x_* - x|( 1 + \frac{1}{\epsilon}|x_* - x|)\,|\nabla w_\epsilon(x)|^m\,I \in \Snn \setminus U.
\end{align*}
Analogously, if $v \in USC(\Omega)$ is a subsolution of \eqref{Eq:FpsiEq} in $\Omega$, and if $v^\epsilon$ is punctually second order differentiable at a point $x \in \Omega$ and $|v^\epsilon(x)| + |v(x^*)| \leq M$, then
\begin{align*}
&F[v^\epsilon](x) +  a|x^* - x|( 1 + \frac{1}{\epsilon}|x^* - x|)\,|\nabla v_\epsilon(x)|^m\,I \in U.
\end{align*}
\end{prop}

\begin{proof}
We only give the proof of the first assertion. The second assertion can be proved in a similar way.

We have 
\begin{equation}
w_{\epsilon}(x+z)\geq w_{\epsilon}(x)+\nabla w_{\epsilon}(x)\cdot z+\frac{1}{2}z^{T}\nabla^{2}w_{\epsilon}(x)z+o(|z|^{2}),\quad\mbox{ as }z\rightarrow 0.\label{yre2}
\end{equation}
By the definition of $w_{\epsilon}$, we have
\begin{equation*}
w_{\epsilon}(x+z)\leq w(x_{*}+z) +\frac{1}{\epsilon}|x_{*}-x|^{2},\label{yree}
\end{equation*}
and therefore, in view of (\ref{yre2}),
\begin{align*}
w(x_{*}+z)&\geq w_{\epsilon}(x+z) - \frac{1}{\epsilon}|x_{*}-x|^{2}\\
&\geq P_{\epsilon}(x_* + z)+o(|z|^{2}),\quad\mbox{ as }z\rightarrow 0,
\end{align*}
where $P_{\epsilon}$ is a quadratic polynomial with 
\begin{align*}
P_{\epsilon}(x_*)&=w_{\epsilon}(x) -\frac{1}{\epsilon}|x_{*}-x|^{2} = w(x_*), \nonumber \\
\nabla P_{\epsilon}(x_*)&=\nabla w_{\epsilon}(x),\\
\nabla^{2} P_{\epsilon}(x_*)&=\nabla^{2} w_{\epsilon}(x).
\end{align*}
Since $w$ is a supersolution of \eqref{Eq:FpsiEq}, we thus have 
\[
\nabla^2 w_\epsilon(x) + L(x_*, w(x_*),\nabla w_{\epsilon}(x)) = F[P_{\epsilon}](x_*)\in \Snn\setminus U.\label{fanyang}
\]
On the other hand, in view of \eqref{Eq:Lcond0}, \eqref{Eq:Lcond1} and $w(x_*) = w_{\epsilon}(x) -\frac{1}{\epsilon}|x_{*}-x|^{2} \leq w_\epsilon(x)$, 
\[
L(x, w_\epsilon(x),\nabla w_{\epsilon}(x)) - L(x_*, w(x_*),\nabla w_{\epsilon}(x)) \leq C(|x - x_*| + \frac{1}{\epsilon}|x - x_*|^2)|\nabla w_\epsilon(x)|^m\,I.
\]
The conclusion is readily seen thanks to \eqref{Eq:UCondPos}.
\end{proof}

%+++++++++++++++++%

\subsection{First variation of $F[\psi]$}

As mentioned in the introduction, we would like to perturb a given function $\psi$ to another function $\tilde\psi$ in such a way that $F[\tilde\psi] $ is bounded from above/below by a multiple of $F[\psi]$ and with a favorable excess term. This will be important in controlling error accrued in other parts of the proof of Theorem \ref{thm:CPNUE} (e.g. in regularizations).

\begin{lem}\label{Lem:FVSub}
Let $\Omega$ be an open bounded subset of $\RR^n$, $n \geq 2$, $L: \Omega \times \RR \times \RR^n \rightarrow \Snn$ be a locally Lipschitz continuous function satisfying \eqref{Eq:Lcond1} and \eqref{Eq:Lcond2} for some $m > 1$, $F$ be given by \eqref{Eq:FDef}, and $\psi: \Omega \rightarrow \RR \cup \{\pm \infty\}$. For any $M > 0$, there exist positive constants $\mu_0, \alpha, \beta, \delta, K_0 > 0$, depending only on an upper bound of $M$, $L$ and $\Omega$, such that 
$$
\mu_0\,\beta\,\sup_\Omega e^{-\beta\psi} \leq \frac{1}{2},
$$
and, for any $0 < \mu < \mu_0$, $\tau \in \RR$, the function $\tilde\psi_{\mu,\tau} = \psi + \mu \,(e^{\alpha|x|^2} + e^{-\beta\psi} - \tau)$ satisfies
\[
F[\tilde\psi_\mu] \geq (1 - \mu\,\beta\,e^{-\beta\psi})F[\psi]
  	+ \mu\,K_0[(1 + |\nabla\psi|^m)\,I + \nabla \psi \otimes \nabla\psi]
\]
in the set
\begin{align}
\Omega^{M,\delta} 
	&:= \Big\{x \in \Omega: \text{$\psi$ is punctually second order differentiable at $x$},\nonumber \\
		&\qquad\qquad |\psi(x)| \leq M, \text{ and }  e^{\alpha|x|^2} + e^{-\beta\psi(x)} - \tau \geq -\delta\Big\}.\label{Eq:OmegaMDef}
\end{align}
\end{lem}

\begin{proof} In the proof, $C$ will denote some large positive constant which may become larger as one moves from lines to lines but depends only on an upper bound for $M$, $L$ and $\Omega$. Eventually, we will choose large $\beta = \beta(C) > 0$, small $\alpha = \alpha(\beta, M, C) > 0$, and finally small $\mu_0 = \mu_0(\alpha, \beta,M, C) > 0$.

We set $\varphi(x) = e^{\alpha|x|^2}$, $f(\psi) = -  e^{-\beta\psi}$ and abbreviate  $\tilde\psi_\mu = \tilde\psi_{\mu,\tau} = \psi + \mu\,(\varphi - f(\psi) - \tau)$. Note that $f'(\psi) > 0$.

We assume in the sequel that $\alpha < 1$, $\delta < 1$ and
\begin{align}
\mu_0 \sup_\Omega [1 + f'(\psi)] &\leq \frac{1}{C} < \frac{1}{2}.
	\label{Eq:mu0Req1}
\end{align}

The following computation is done at a point in $\Omega^{M,\delta}$. We have
\begin{align*}
F[\tilde\psi_\mu]
	&\geq (1 - \mu\,f'(\psi)) F[\psi] - \mu\,f''(\psi) \nabla\psi \otimes \nabla\psi
		+ 2\mu \alpha\,\varphi\, I
		\nonumber\\
		&\qquad\qquad + L(x,\tilde\psi_\mu,\nabla \tilde\psi_\mu) - (1 - \mu\,f'(\psi))  L(x,\psi,\nabla\psi).
\end{align*}
Noting that $\varphi - f(\psi) - \tau \geq -\delta$ in $\Omega^{M,\delta}$, we deduce from \eqref{Eq:Lcond1} and \eqref{Eq:mu0Req1} that
\[
L(x,\tilde\psi_\mu,\nabla \tilde\psi_\mu) \geq L(x,\psi,\nabla \tilde\psi_\mu) - C\,\mu\,\delta\,(|\nabla\psi|^m + \mu^m\,\alpha^m\,\varphi^m)\,I.
\]
Therefore,
\begin{align}
F[\tilde\psi_\mu]
	&\geq (1 - \mu\,f'(\psi)) F[\psi] - \mu\,f''(\psi) \nabla\psi \otimes \nabla\psi\nonumber\\
		&\qquad\qquad + 2\mu \alpha\,(1 - C\delta\mu^m\alpha^{m-1}\varphi^{m-1})\varphi\, I - C\,\mu\,\delta|\nabla\psi|^m\,I
		\nonumber\\
		&\qquad\qquad + L(x,\psi,\nabla \tilde\psi_\mu) - (1 - \mu\,f'(\psi))  L(x,\psi,\nabla\psi).
		\label{Eq:FtpsiEst1}
\end{align}

We proceed to estimate $L(x,\psi,\nabla \tilde\psi_\mu) - (1 - \mu\,f'(\psi))  L(x,\psi,\nabla\psi)$. For $0 \leq t \leq \mu$, let
\[
g(t) = \frac{L(x,\psi,\nabla \tilde\psi_t)}{1 - t f'(\psi)}. 
\]
We have
\begin{align*}
\frac{d}{dt} g(t)
	&\geq \frac{f'(\psi)}{(1 - t f'(\psi))^2} \Big[L(x,\psi,\nabla \tilde\psi_t) - \nabla\tilde\psi_t \cdot \nabla_p L(x,\psi,\nabla\tilde\psi_t)\\
		&\qquad\qquad - \frac{C\alpha \varphi}{f'(\psi)} |\nabla_p L(x,\psi,\nabla\tilde\psi_t)|\,I \Big].
\end{align*}
Thus, in view of \eqref{Eq:Lcond2} and \eqref{Eq:mu0Req1}, if $\alpha, \beta$ and $\delta$ satisfy
\begin{align}
\alpha \sup_\Omega \varphi[\frac{1}{f'(\psi)} + 1]\,&\leq \frac{1}{C},\label{Eq:alphaReq}
\end{align}
then, with $\Lambda = 8C$ and $\theta = \frac{\alpha \varphi}{8f'(\psi)}$ in \eqref{Eq:Lcond2},
\begin{align*}
\frac{d}{dt} g(t)
	&\geq f'(\psi) \Big[-C \nabla\tilde\psi_t \otimes \nabla \tilde\psi_t + \frac{1}{C} |\nabla\tilde\psi_t|^m\,I\Big] -  \frac{1}{2}\alpha\,\varphi \,I\\
	&\geq f'(\psi) \Big[-C \nabla\psi \otimes \nabla \psi + \frac{1}{C} |\nabla\psi|^m\,I\Big] -  \alpha\,\varphi \,I.
\end{align*}
This implies
\begin{align}
&L(x,\psi,\nabla \tilde\psi_\mu) - (1 - \mu\,f'(\psi))  L(x,\psi,\nabla\psi)\nonumber\\
	&\qquad\qquad= (1 - \mu\,f'(\psi))[ g(\mu) - g(0)] \nonumber\\
	&\qquad\qquad\geq \mu\,f'(\psi) \Big[-C \nabla\psi \otimes \nabla \psi + \frac{1}{C} |\nabla\psi|^m\,I\Big] - \mu\,\alpha\,\varphi\,I.
	\label{Eq:FtpsiEst2}
\end{align}

Combining \eqref{Eq:FtpsiEst1} and \eqref{Eq:FtpsiEst2} and using \eqref{Eq:alphaReq}, we obtain
\begin{align}
F[\tilde\psi_\mu]
	&\geq (1 - \mu\,f'(\psi)) F[\psi]  + \mu\, \alpha\,\varphi I
		+  \frac{1}{C}\, \mu\,(f'(\psi) - C\delta)  |\nabla\psi|^m\,I
		\nonumber\\
		&\qquad\qquad + \mu\,\big[-f''(\psi)  - Cf'(\psi)\big]\,\nabla\psi \otimes \nabla\psi 	.
		\label{Eq:FtpsiEst2Y}
\end{align}

We now fix $C$ and proceed with the choice of $\alpha, \beta, \delta$ and $\mu_0$. First, choosing $\beta \geq 2C$ and recalling the definition of $f$, we have
\[
-f''(\psi)  - Cf'(\psi) = \beta(\beta - C)e^{-\beta \psi} \geq \frac{1}{2}\beta\,f'(\psi).
\]
Next, choose $\alpha$ such that \eqref{Eq:alphaReq} is satisfied and choose $\delta$ such that $f'(\psi) - C\delta \geq \frac{1}{2}f'(\psi)$. Finally, choose $\mu_0$ such that \eqref{Eq:mu0Req1} holds. We hence obtain from \eqref{Eq:FtpsiEst2Y}  that
\[
F[\tilde\psi]
	\geq (1 - \mu\,f'(\psi)) F[\psi]
		+ \mu  \,\alpha\,\varphi\, I + \frac{1}{C}\, \mu\,f'(\psi)  |\nabla\psi|^m\,I + \frac{1}{2}\beta\,\mu\,f'(\psi) \nabla \psi \otimes \nabla \psi.
\]
This completes the proof.
\end{proof}

\begin{lem}
Let $\Omega$ be an open bounded subset of $\RR^n$, $n \geq 2$, $L: \Omega \times \RR \times \RR^n \rightarrow \Snn$ be a locally Lipschitz continuous function satisfying \eqref{Eq:Lcond1} and \eqref{Eq:Lcond2} for some $m > 1$, $F$ be given by \eqref{Eq:FDef}, and $\psi: \Omega \rightarrow \RR \cup \{\pm \infty\}$. There exist positive constants $\mu_0, \alpha, \beta, \delta, K_0 > 0$, depending only on an upper bound of $\sup_\Omega |\psi|$, $L$ and $\Omega$, such that, for any $0 < \mu < \mu_0$, $\tau \in \RR$, the function $\hat\psi_{\mu,\tau} = \psi - \mu \,(e^{\alpha|x|^2} + e^{-\beta\psi} - \tau)$ satisfies
\[
F[\hat\psi_\mu] \leq (1 + \mu\,\beta\,e^{-\beta\psi})F[\psi]
  	- \mu\,K_0[(1 + |\nabla \psi|^m)\,I + \nabla\psi \otimes \nabla \psi]
\]
in the set $\Omega^{M,\delta}$ defined by \eqref{Eq:OmegaMDef}.
\end{lem}

\begin{proof} The proof is similar to that of Lemma \ref{Lem:FVSub} and is omitted.
\end{proof}

%+++++++++++++++++%

\subsection{Proof of Theorem \ref{thm:CPNUE}}

Arguing by contradiction, we suppose that there exists $\gamma>0$ such that 
\begin{equation*}
\max\limits_{\bar\Omega}(v-w) = 0\quad  \text{ and }\quad (v-w)(x)\leq-\gamma,\quad\forall x\in\overline{\Omega\setminus \Omega_{\gamma}}\label{lpl}
\end{equation*} 
where $\Omega_\gamma = \{x \in \Omega: \textrm{dist}(x, \partial\Omega) > \gamma\}$. 

For $\epsilon > 0$, let $v^\epsilon$ and $w_\epsilon$ be the $\epsilon$-upper and $\epsilon$-lower envelops of $v$ and $w$ respectively (see Section \ref{Sec:Prelim}). We note that
\[
v \leq v^\epsilon \leq \max_{\bar\Omega} v < +\infty \text{ and } w \geq w_\epsilon \geq \min_{\bar\Omega} w > -\infty. 
\]

In the sequel, we use $C$ to denote some positive constant which depends on $\max_{\bar\Omega} v$, $\min_{\bar\Omega} w$, $L$ and $\Omega$ but is always independent of $\epsilon$.

By Lemma \ref{Lem:FVSub}, we can find $\bar \mu > 0$, $\delta > 0$ and a smooth positive function $f: \RR^n \times \RR \rightarrow (0,\infty)$, depending only on $\max_{\bar\Omega} v$, $\min_{\bar\Omega} w$, $L$ and $\Omega$, such that $f$ is decreasing with respect to the $s$-variable, $\bar\mu \sup_\Omega |\partial_s f(\cdot, v^\epsilon)| \leq\frac{1}{2}$ and, for $\mu \in (0,\bar\mu)$, $\tau \in \RR$ and $\tilde v_{\epsilon,\tau} = v^\epsilon + \mu (f(\cdot,v^\epsilon) - \tau)$, there holds
\begin{equation}
F[\tilde v_{\epsilon,\tau}] \geq (1 - \mu|\partial_s f(\cdot, v^\epsilon)|)F[v^\epsilon]
  	+ \frac{\mu}{C}(1 + |\nabla v^\epsilon|^m)\,I
			\label{Eq:Ftw}
\end{equation}
in the set
\begin{multline*}
\tilde\Omega_{\epsilon} := \Big\{x \in \Omega_{\gamma/2}: \text{$v^\epsilon$ is punctually second order differentiable at $x$},\\
	v^\epsilon(x) \geq \min_{\bar\Omega} w- 1 \text{ and } f(x,v^\epsilon(x)) - \tau \geq -\delta \Big\}.
\end{multline*}
Note that $\bar\mu$ and $\delta$ are independent of $\epsilon$. Furthermore, in view of \eqref{Eq:UpLowConv}, there exists $\bar \eta > 0$ independent of $\epsilon$ such that, for all small $\epsilon$ and $\eta \in (0,\bar\eta)$, one can (uniquely) find $\tau = \tau(\epsilon,\eta)$ such that the function $\xi_{\epsilon,\eta} := \tilde v_{\epsilon,\tau} - w_\epsilon$ satisfies
\[
\max_{\bar\Omega} \xi_{\epsilon,\eta} = \eta \text{ and } \xi_{\epsilon,\eta} < -\frac{\gamma}{2} \text{ in }\overline{\Omega\setminus \Omega_{\gamma}}.
\]

Let $\Gamma_{\xi_{\epsilon,\eta}^{+}}$ denote the concave envelope of $\xi_{\epsilon,\eta}^{+}:=\max\{\xi_{\epsilon,\eta},0\}$ on $\bar\Omega$. Then by \eqref{utr}, we have 
\begin{equation*}
\nabla^{2}\xi_{\epsilon,\eta}\geq-\frac{4}{\epsilon}I\quad\mbox{ a.e. in }\Omega_{\gamma}.
\end{equation*}
By \cite[Lemma 3.5]{CabreCaffBook}, we have 
\begin{equation*}
\int_{\{\xi_{\epsilon,\eta}=\Gamma_{\xi_{\epsilon,\eta}^{+}}\}}\mbox{det}(-\nabla^{2}\Gamma_{\xi_{\epsilon,\eta}^{+}})>0,
\end{equation*}
which implies that the Lebesgue measure of $\{\xi_{\epsilon,\eta}=\Gamma_{\xi_{\epsilon,\eta}^{+}}\}$ is positive. Then there exists $x_{\epsilon,\eta}\in\{\xi_{\epsilon,\eta}=\Gamma_{\xi_{\epsilon,\eta}^{+}}\}\cap\Omega_{\gamma}$ such that both of $v^\epsilon$ and $w_\epsilon$ are punctually second order differentiable at $x_{\epsilon,\eta}$, 
\begin{equation}
0<\xi_{\epsilon,\eta}(x_{\epsilon,\eta})\leq\eta,\label{Eq:29Dec16b}
\end{equation}
\begin{equation}
|\nabla\xi_{\epsilon,\eta}(x_{\epsilon,\eta})| = |\nabla \tilde v_{\epsilon,\tau}(x_{\epsilon,\eta})- \nabla w_{\epsilon}(x_{\epsilon,\eta})| \leq C\eta,\label{Eq:29Dec16c}
\end{equation}
and
\begin{equation}
\nabla^{2}\xi_{\epsilon,\eta}(x_{\epsilon,\eta})=\nabla^2 \tilde v_{\epsilon,\tau}(x_{\epsilon,\eta})- \nabla^2 w_{\epsilon}(x_{\epsilon,\eta})\leq 0.\label{Eq:29Dec162m}
\end{equation}

From \eqref{Eq:29Dec16b} and the definition of $\tilde v_{\epsilon,\tau}$, we have
\begin{equation}
	f(x_{\epsilon,\eta},v^\epsilon(x_{\epsilon,\eta})) - \tau 
	> \frac{1}{\mu}(w_\epsilon(x_{\epsilon,\eta}) - v^\epsilon(x_{\epsilon,\eta})).
	\label{Eq:f-tau->0}
\end{equation}
Note that, as $w \geq v$ in $\Omega$, Lemma \ref{lem:EquiSC} implies that
\[
\liminf_{\epsilon \rightarrow 0, \eta \rightarrow 0}[w_\epsilon(x_{\epsilon,\eta}) - v^\epsilon(x_{\epsilon,\eta})] \geq 0.
\]
Hence, by shrinking $\mu$ and $\bar\eta$ if necessary, we may assume for all small $\epsilon$ that
\[
f(x_{\epsilon,\eta},v^\epsilon(x_{\epsilon,\eta})) - \tau \geq -\delta, \qquad v^\epsilon(x_{\epsilon,\eta}) \geq \min_{\bar\Omega} w - 1, \quad \text{ and } w_\epsilon(x_{\epsilon,\eta}) \leq \max_{\bar\Omega} v + 1.
\]
We deduce that $x_{\epsilon,\eta} \in \tilde\Omega_{\epsilon,\delta}$ and thus obtain from \eqref{Eq:Ftw} that
\begin{equation}
F[\tilde v_{\epsilon,\tau}](x_{\epsilon,\eta}) \geq (1 - \mu|\partial_s f(x_{\epsilon,\eta}, v^\epsilon(x_{\epsilon,\eta}))|)F[v^\epsilon](x_{\epsilon,\eta})
  	+ \frac{\mu}{C}(1 + |\nabla v^\epsilon(x_{\epsilon,\eta})|^m)\,I.
		\label{Eq:FtwInAction}
\end{equation}

Next, the proof of \eqref{Eq:UpLowGradEst} implies that, for any unit vector $e$,
\[
\partial_e v^{\epsilon}(x_{\epsilon,\eta}) \geq - \frac{C}{\sqrt{\epsilon}} \text{ and } \partial_e w_{\epsilon}(x_{\epsilon,\eta}) \leq \frac{C}{\sqrt{\epsilon}}.
\]
This together with \eqref{Eq:29Dec16c} implies that, for all sufficiently small $\eta$,
\[
|\nabla \tilde v_{\epsilon,\tau}(x_{\epsilon,\eta})| + |\nabla w_{\epsilon}(x_{\epsilon,\eta})| \leq \frac{C}{\sqrt{\epsilon}}.
\]
Thus, by the local Lipschitz regularity of $L$, 
\[
L(x_{\epsilon,\eta}, w_\epsilon(x_{\epsilon,\eta}),\nabla w_{\epsilon}(x_{\epsilon,\eta})) - L(x_{\epsilon,\eta}, \tilde v_{\epsilon,\tau}(x_{\epsilon,\eta}), \nabla \tilde v_{\epsilon}(x_{\epsilon,\eta}))
	\geq -C(\epsilon)\eta\,I.
\]
This together with \eqref{Eq:29Dec162m} implies that
\begin{equation}
F[w_{\epsilon}](x_{\epsilon,\eta}) \geq F[\tilde v_{\epsilon,\tau}](x_{\epsilon,\eta}) - C(\epsilon)\,\eta\,I.
	\label{Eq:CPFwtv>}
\end{equation}
Recalling \eqref{Eq:FtwInAction}, we can find $\hat\eta = \hat\eta(\epsilon)$ such that, for $0 < \eta < \hat \eta(\epsilon)$, there holds
\begin{equation}
F[w_{\epsilon}](x_{\epsilon,\eta}) \geq (1 - \mu|\partial_s f(x_{\epsilon,\eta}, v^\epsilon(x_{\epsilon,\eta}))|)F[v^\epsilon](x_{\epsilon,\eta})
  	+ \frac{\mu}{C}(1 + |\nabla v^\epsilon(x_{\epsilon,\eta})|^m)\,I.
		\label{Eq:Fweve>}
\end{equation}

We next claim that 
\begin{equation}
\liminf_{\epsilon \rightarrow 0, \eta \rightarrow 0} \frac{1}{\epsilon}\Big[|(x_{\epsilon,\eta})_* - x_{\epsilon,\eta}|^2 + |(x_{\epsilon,\eta})^* - x_{\epsilon,\eta}|^2\Big] \leq C\mu^2.
	\label{Eq:30Rep}
\end{equation}
Assuming this claim for now, we use Proposition \ref{key lemma} to find $a > 0$ independent of $\epsilon$ and $\eta$ such that one has, in $\Omega_\gamma$,
\begin{align}
F[w_{\epsilon}](x_{\epsilon,\eta})- a|(x_{\epsilon,\eta})_* - x_{\epsilon,\eta}|(1+\frac{1}{\epsilon}|(x_{\epsilon,\eta})_* - x_{\epsilon,\eta}|)\, |\nabla w_\epsilon(x_{\epsilon,\eta})|^m\,I
	&\in \Snn\setminus U,
	\label{Eq:Fwe}\\
F[v^{\epsilon}](x_{\epsilon,\eta}) + a|(x_{\epsilon,\eta})^* - x_{\epsilon,\eta}|(1 + \frac{1}{\epsilon}|(x_{\epsilon,\eta})^* - x_{\epsilon,\eta}|)\, |\nabla v^\epsilon(x_{\epsilon,\eta})|^m\,I
	&\in \overline{U},
	\label{Eq:Fve}
\end{align}
where $x_*$ and $x^*$ are as in Section \ref{Sec:Prelim}. The relations \eqref{Eq:Fweve>}, \eqref{Eq:Fwe} and \eqref{Eq:Fve} amount to a contradiction for sufficiently small $\mu$ thanks to \eqref{Eq:UCondPos} and \eqref{Eq:UCond*S}. Therefore, to conclude the proof it suffices to prove the claim \eqref{Eq:30Rep}.

Pick some $\eta(\epsilon) < \hat\eta(\epsilon)$ such that $\eta(\epsilon) \rightarrow 0$ as $\epsilon \rightarrow 0$. Pick a sequence $\epsilon_m \rightarrow 0$ such that, for $x_m := x_{\epsilon_m,\eta(\epsilon_m)}$, the sequence $\frac{1}{\epsilon_m}[|(x_m)^* - x_m|^2 + |(x_m)_* - x_m|^2]$ converges to a limit which we will show to be no larger than $C\mu^2$. We will abbreviate $\tau_m := \tau(\epsilon_m, \eta(\epsilon_m))$, $v^m = v^{\epsilon_m}$, $w_m = w_{\epsilon_m}$. Without loss of generality, we may also assume that $x_m \rightarrow x_0 \in \Omega$, $f(x_m, v^m(x_m)) \rightarrow f_0$ and $\tau_m \rightarrow \tau_0$. 

As $\max_{\bar\Omega} \xi_{\epsilon,\eta} = \eta$, we have in view of \eqref{Eq:UpLowConv} that
\begin{multline}
v(x_0) - w(x_0) + \mu(f(x_0, v(x_0)) -\tau_0)  \\
	 = \lim_{m \rightarrow \infty} \big\{v^{m}(x_0) - w_{m}(x_0) + \mu(f(x_0, v^{m}(x_0)) -\tau_m) \big\} \leq 0.
	 \label{Eq:vwmftLim}
\end{multline}
On the other hand, by \eqref{Eq:29Dec16b} and the fact that $f$ is decreasing in $s$, we have
\begin{align*}
f(x_0, \limsup_{m \rightarrow \infty} v^m(x_m))
	&\leq  f_0 
		= \lim_{m \rightarrow \infty} f(x_{m}, v^{m}(x_m))\\
	& \leq \limsup_{m \rightarrow \infty}  f(x_m, w_m(x_m) - \mu(f(x_m, v^m(x_m)) - \tau_m))\\
	& \leq   f(x_0, \liminf_{m \rightarrow \infty} w_m(x_m) - \mu(f_0 - \tau_0)),
\end{align*}
which implies, in view of Lemma \ref{lem:EquiSC} and the fact that $w \geq v$, that 
\[
f(x_0, w(x_0)) 
	\leq f(x_0, v(x_0)) 
	\leq f_0
	 \leq   f(x_0, w(x_0) - \mu(f_0 - \tau_0)),
\]
which further implies that
\[
0 \leq f_0 - f(x_0, v(x_0)) \leq C\mu.
\]
Together with \eqref{Eq:vwmftLim}, this implies that
\[
v(x_0) - w(x_0) + \mu(f_0 -\tau_0)  \leq C\mu^2.
\]

We are now ready to wrap up the argument. As $(x_\epsilon)^* - x_\epsilon \rightarrow 0$ and $(x_\epsilon)_* - x_\epsilon \rightarrow 0$ as $\epsilon \rightarrow 0$, we have $(x_m)_* \rightarrow x_0$ and $(x_m)^* \rightarrow x_0$. As $v$ is upper semi-continuous and $w$ is lower semi-continuous, we have
\[
\limsup_{m \rightarrow \infty} v((x_m)^*) \leq v(x_0) \text{ and } \liminf_{m \rightarrow \infty} w((x_m)_*) \geq w(x_0).
\]
Thus, by the left half of \eqref{Eq:29Dec16b},
\begin{align*}
0 
	&\leq \limsup_{m \rightarrow \infty} \frac{1}{\epsilon_m}[|(x_{\epsilon_m})^* - x_{\epsilon_m}|^2 + |(x_{\epsilon_m})_* - x_{\epsilon_m}|^2]\\
	&\leq \limsup_{m \rightarrow \infty} \big\{v((x_{\epsilon_m})^*) - w((x_{\epsilon_m})_*) + \mu(f(x_{\epsilon_m},v^{\epsilon_m}(x_{\epsilon_m})) - \tau({\epsilon_m},\eta({\epsilon_m}))]\big\}\\
	&\leq v(x_0) - w(x_0) + \mu(f_0 - \tau_0) \leq C\mu^2.
\end{align*}
We have proved \eqref{Eq:30Rep}, and thus concluded the proof.
\hfill$\Box$

\subsection{Proof of Corollary \ref{cor:CPSing}}

We will use a result from \cite{CafLiNir11}.

\begin{thm}[\cite{CafLiNir11}]\label{thm:CLNRemSing}
Let $n \geq 1$, $\Omega \subset \RR^n$ be a non-empty open set,
%,  $(a_{ij}), b_i$ and $c$ be as above. 
 and let $a,b \in C^0(\Omega \times \RR \times \RR^n)$ satisfy
$$a(x,z,p) \geq 0 \text{ for all } x \in \Omega, z \in \RR, p \in \RR^n,$$
 and $U \subset \Snn$ be a non-empty open set satisfying \eqref{Eq:UCondPos}. If $u \in LSC(\bar \Omega)$ satisfies
 \[
 \Delta u \leq C  \text{ in } \Omega \text{ in the viscosity sense}, 
 \]
and, for some subset $E$ of $\Omega$ of zero Lebesgue measure,
\[
a(x,u,Du)\nabla^2 u + b(x,u,\nabla u) \in \Snn \setminus U \text{ in } \Omega \setminus E \text{ in the viscosity sense},
\]
then
\[
a(x,u,Du)\nabla^2 u + b(x,u,\nabla u) \in \Snn \setminus U \text{ in } \Omega \text{ in the viscosity sense}.
\]
\end{thm}

\begin{rem}
This result was not stated as above in \cite{CafLiNir11}. However, the proof of \cite[Theorem 1.2]{CafLiNir11} in effect yields the above result.
\end{rem}

\begin{proof}[Proof of Corollary \ref{cor:CPSing}]
Note that constant functions are solutions of \eqref{Eq:FpsiEq} and the max of two subsolutions is a subsolution. It thus suffices to consider the case when
\[
\inf_{\bar\Omega} v > -\infty.
\]

By \eqref{gooo},
\[
\Delta v + (\alpha - n\beta)|\nabla v|^2 \geq 0 \text{ in } \Omega \setminus E \text{ in the viscosity sense}.
\]

Note that when \eqref{Eq:anb2} holds, then the function $\tilde u = e^{-\frac{1}{|\alpha - n\beta|}v}$ satisfies
\[
\Delta \tilde u \leq 0 \text{ in } \Omega \setminus E \text{ in the viscosity sense}.
\]
As $\tilde u > 0$ in $\Omega\setminus E$ and $E$ has zero capacity, the maximum principle then implies that $\tilde u > \frac{1}{c} > 0$ in $\bar \Omega \setminus E$, and hence $\sup_{\bar\Omega \setminus E} v < +\infty$. Thus we can assume without loss of generality that \eqref{Eq:anb1} holds.

In view of the comparison principle Theorem \ref{thm:CPQuad}(b), it suffices to show that
\begin{equation}
F[v] \in \bar U \text{ in } \Omega \text{ in the viscosity sense},
	\label{Eq:CPSv+E}
\end{equation}
where we define, for $x \in E$,
\[
v(x) = \limsup_{y \rightarrow x, y \in \Omega \setminus E} v(y) < +\infty.
\]
Indeed, we note that, for $C > |\alpha - n\beta|$, the function $u = -e^{Cv} \in LSC(\bar\Omega)$ satisfies $\inf_{\bar\Omega} u > -\infty$, $\sup_{\bar\Omega} u < 0$ and
\[
\Delta u = Cu(\Delta v + C|\nabla v|^2) \leq 0 \text{ in } \Omega \setminus E \text{ in the viscosity sense}.
\]
Since $E$ has zero capacity, it follows that 
\[
\Delta u \leq 0 \text{ in } \Omega \text{ in the viscosity sense}.
\]
An application of Theorem \ref{thm:CLNRemSing} (to the set $\tilde U = \Snn \setminus (-\bar U)$) then implies that
\[
F[v] = \frac{1}{C}\frac{\nabla^2 u}{u} - \frac{1}{C} \frac{\nabla u \otimes \nabla u}{u^2} + L\Big(x,\frac{1}{C}\ln(-u), \frac{\nabla u}{Cu}\Big) \in \bar U \text{ in } \Omega \text{ in the viscosity sense},
\]
which proves \eqref{Eq:CPSv+E}, and hence the assertion.
\end{proof}

%+++++++++++++++++%

\subsection{Counterexamples to the propagation principle}\label{Sec:CPCounterex}

It this section, we give examples to illustrate that \eqref{Eq:UCondPos}, i.e. degenerate ellipticity, the properness and regularity of $L$ is insufficient to ensure the correctness of the propagation principle. These examples will also illustrate the importance of various technical assumptions in Theorems \ref{thm:CPQuad} and \ref{thm:CPNUE}.

Let $a, b \in C^1_{loc}([0,\infty))$ and consider for now a rotationally invariant operator $F$ of the form
\begin{equation}
F[\psi] = \nabla^2 \psi + a(|\nabla \psi|)\nabla\psi \otimes \nabla \psi + b(|\nabla\psi|)I.
	\label{Eq:FpiRotInv}
\end{equation}
In other words, we have
\[
L(p) = a(|p|) p \otimes p - b(|p|)\,I.
\]
Note that although $a, b$ are locally differentiable, $L$ is in general only locally Lipschitz. $L$ is locally differentiable if and only if $b'(0) = 0$.

The following example suggests that some delicate attention should be paid if one allows $m = 1$ in condition \eqref{Eq:Lcond2} (in the context of Theorem \ref{thm:CPNUE}).

\begin{prop}\label{prop:Ctex}
Let $a, b \in C^1([0,\infty))$ and $F$ be of the form \eqref{Eq:FpiRotInv}. If
\[
b(0) = 0 \text{ and } b'(0) \neq 0,
\]
then the propagation principle does not hold for $F$, namely there exist a bounded domain $\Omega \in \RR^n$, a non-empty open set $U \subset \Snn$ satisfying \eqref{Eq:UCondPos}, and a supersolution $w \in C^2(\bar\Omega)$ and a subsolution $v \in C^2(\bar\Omega)$ of \eqref{Eq:FpsiEq} in $\Omega$ such that $w > v$ on $\partial\Omega$, but $\min_{\bar\Omega} (w - v) = 0$.
\end{prop}

\begin{proof}
Considering $F[-\psi]$ instead of $F[\psi]$ if necessary, we can assume without loss of generality that 
\begin{equation}
b'(0) < 0. 
	\label{Eq:b'0<0}
\end{equation}

Let $U$ be the set of positive definite symmetric $n \times n$ matrices. Note that $v \equiv 0$ is a solution of \eqref{Eq:FpsiEq} on $\RR^n$. Since $L$ is independent of $\psi$, it suffices to exhibit a bounded domain $\Omega$, and a supersolution $w \in C^2(\bar\Omega)$ of \eqref{Eq:FpsiEq} in $\Omega$ such that $w > 0$ on $\partial\Omega$, but $\min_{\bar\Omega} w = 0$.

In view of \eqref{Eq:b'0<0} and the fact that $b(0) = 0$, there exists some $r_0 > 0$ and $\delta > 0$ such that
\begin{equation}
b(s) < 0 \text{ and } r_0 > \frac{s}{|b(s)|} \text{ for all } s \in (0,\delta).
	\label{Eq:Ctexr0}
\end{equation}
Let $\Omega = \{r_0 -1 < |x| < r_0 + 1\}$ and $w(x) = w(|x|)$ for some $w \in C^2([r_0 - 1 ,r_0 + 1])$ satisfying $w(r_0) = w'(r_0) = 0$ and
\begin{align}
w'(r) &\in (-\delta,0) \text{ for } r \in [r_0-1,r_0),\label{Eq:Ctexw'left}\\
w'(r) &\in (0,\delta) \text{ for } r \in (r_0, r_0+1].\label{Eq:Ctexw'right}
\end{align}
Then $w > 0$ on $\partial\Omega$ and $\min_{\bar\Omega} w = 0$.

To conclude the proof, we check that $F[w]$ is not positive definite. Indeed, the eigenvalues of $F[w]$ are given by
 $$
 \lambda(F[w]) = (\mu, \nu, \ldots, \nu) \text{ where } \mu = w'' + a(|w'|)|w'|^2 + b(|w'|) \text{ and } \nu = \frac{1}{r}w' + b(|w'|).
$$
Now, for $r < r_0$, we have $w' < 0$ (thanks to \eqref{Eq:Ctexw'left}) and $b(|w'|) < 0$ and so $\nu < 0$. For $r > r_0$, we have, in view of \eqref{Eq:Ctexr0} and \eqref{Eq:Ctexw'right},
\[
\nu = w'\Big(\frac{1}{r} - \frac{|b(w')|}{w'}\Big) < w'\Big(\frac{1}{r} - \frac{1}{r_0}\Big) < 0.
\]
Also, $\nu = 0$ when $r = r_0$. It thus follows that $\nu \leq 0$ in $(r_0 - 1, r_0 + 1)$, i.e. $F[w]$ is not positive definite. The proof is complete.
\end{proof}

The previous result show that the propagation principle does not hold for general operators of the form \eqref{Eq:FDef}. However, it is easy to see that the function $L$ in Proposition \ref{prop:Ctex} is Lipschitz but not $C^1$. We will next construct some counterexamples with smooth $L$.

For $\alpha \in \RR$, consider the rotationally invariant operator
\begin{equation}
F[\psi] = \nabla^2 \psi - (\psi^3\,|\nabla\psi|^{10} + \alpha\,\psi\,|\nabla\psi|^6 + |\nabla \psi|^4)I,
	\label{Eq:F2/3Form}
\end{equation}
i.e.
\[
L(s,p) = -(s^3|p|^{10} + \alpha\,s |p|^6 + |p|^4)I,
\]
which is an analytic function of $s$ and $p$. Note that neither condition \eqref{Eq:Lcond2} nor condition \eqref{Eq:Lcond2Cat} is satisfied for this function $L$. Note also that the leading part of $L(s,p)$ changes sign as $s$ varies -- this should be compared the assumption that $\beta(w)$ is of one sign in Theorem \ref{thm:CPQuad}.

\begin{prop}\label{prop:beta<>0}
Let $n \geq 2$, $U$ be the set of positive definite symmetric $n \times n$ matrices, and $F$ be of the form \eqref{Eq:F2/3Form} for some $\alpha < -\frac{5}{2}$. Then the propagation principle does not hold: there exists a bounded domain $\Omega$, a supersolution $w$ and a subsolution $v$ of \eqref{Eq:FpsiEq} in $\Omega$ such that $w > v$ on $\partial\Omega$ but $\min_{\bar\Omega} (w - v) = 0$.
\end{prop}

\begin{proof}
Fix some $r_0 > 0$. For $t \in \RR$, let
\begin{equation}
\psi_t(x) = \psi_t(r) = t^{\frac{1}{3}}\,|r - r_0|^{\frac{2}{3}}, \qquad \text{ where }r = |x|.
	\label{Eq:psitbeta<>0}
\end{equation}
The eigenvalues of $F[\psi_t]$ are $(\lambda_{1,t}, \lambda_{2,t}, \ldots, \lambda_{2,t})$ where
\begin{align*}
\lambda_{1,t}
	&= \psi_t'' - \psi_t^3|\psi_t'|^{10} - \alpha\,\psi_t|\psi_t'|^6 - |\psi_t'|^4\\
	&= - \frac{2}{59049}\frac{t^{\frac{1}{3}}(8P_4(t) + 6561)}{|r-r_0|^{\frac{4}{3}}},\\
\lambda_{2,t}
	&= \frac{1}{r} \psi_t' - \psi_t^3|\psi_t'|^{10} - \alpha\,\psi_t|\psi_t'|^6 - |\psi_t'|^4\\
	&= -\frac{2}{59049}\frac{t^{\frac{1}{3}}(8 P_4(t) -19683\frac{r-r_0}{r})}{|r-r_0|^{\frac{4}{3}}},
\end{align*}
and where $P_4(t) = 64\,t^4 + 324\,\alpha\,t^2 + 729\,t$.

Note that $P_4(0) = 0$, and, as $\alpha < -\frac{5}{2}$, 
\begin{align*}
P_4(-2) 
	&= -434 + 1296\,\alpha<  -3674,\\
P_4\Big(\frac{9}{4}\Big)
	&= \frac{6561}{4}(2 + \alpha) < - \frac{6561}{8}.
\end{align*}
It follows that the equation $8P_4(t) + 6561 = 0$ has four roots $t_1, \ldots, t_4$ satisfying $t_1 < -2 < t_2 < 0 < t_3 < \frac{9}{4} < t_4$. In particular, we have $\lambda_{1,t_i}(r) = 0$ for $r \neq r_0$, $i = 1, \ldots, 4$. Also, from the expression of $\lambda_{2,t}$, we can find some small $\delta > 0$ such that
\[
t_i\,\lambda_{2,t_i}(r) > 0 \text{ for } r \neq r_0, |r - r_0| \leq \delta.
\]
In addition, there exists $t_0 < t_1$ such that 
\[
\lambda_{1,t_0}(r) >0 \text{ and } \lambda_{2,t_0}(r) > 0\text{ for } r \neq r_0, |r - r_0| \leq \delta.
\]

We define, 
\begin{align*}
w(x) &= w(r) = \left\{\begin{array}{ll}
	\psi_{t_4}(r) & \text{ for } r_0 \leq r \leq r_0 + \delta,\\
	\psi_{t_2}(r) & \text{ for } r_0 -\delta \leq r < r_0,
\end{array}\right.\\
v(x) &= v(r) = \left\{\begin{array}{ll}
	\psi_{t_3}(r) & \text{ for } r_0 \leq r \leq r_0 + \delta,\\
	\psi_{t_0}(r) & \text{ for } r_0 -\delta \leq r < r_0.
\end{array}\right.
\end{align*}
It is readily seen that $w$ and $v$ are respectively a supersolution and a subsolution of \eqref{Eq:FpsiEq} in $\Omega = \{|r - r_0| < \delta\}$, $w \geq v$ in $\Omega$ and $\{w = v\} = \{r = r_0\}$. This finishes the proof.
\end{proof}

The previous example can be modified to give a counterexample to the propagation principle with $L$ being non-increasing in $s$.

\begin{prop}\label{prop:CtexNonDec}
Let $n \geq 2$ and $U$ be the set of positive definite symmetric $n \times n$ matrices. There exists a smooth function $L: \RR \times \RR^n \rightarrow \mathcal{S}^{n \times n}$ such that $L$ is non-increasing in $s$ but the propagation principle does not hold for $F$ of the form \eqref{Eq:FDef}: there exists a bounded domain $\Omega$, a supersolution $w$ and a subsolution $v$ of \eqref{Eq:FpsiEq} in $\Omega$ such that $w > v$ on $\partial\Omega$ but $\min_{\bar\Omega} (w - v) = 0$.
\end{prop}

\begin{proof}
For $\alpha \in \RR$ to be fixed, consider
\[
\tilde L(s,p) = -(s^3|p|^{10} + \alpha\,s |p|^6 + \frac{1}{100}|p|^4)I.
\]
We first show that the propagation principle does not hold for $\tilde F = \nabla^2 + \tilde L$ as in the proof of Proposition \ref{prop:beta<>0}. 

Fix some $r_0 > 0$. For $t \in \RR$, define $\psi_t$ by \eqref{Eq:psitbeta<>0}. The eigenvalues of $\tilde F[\psi_t]$ are $(\lambda_{1,t}, \lambda_{2,t}, \ldots, \lambda_{2,t})$ where
\begin{align*}
\lambda_{1,t}
	&= \psi_t'' - \psi_t^3|\psi_t'|^{10} - \alpha\,\psi_t|\psi_t'|^6 - \frac{1}{100}|\psi_t'|^4\\
	&= - \frac{2}{1476225}\frac{t^{\frac{1}{3}}(2\tilde P_4(t) + 164025)}{|r-r_0|^{\frac{4}{3}}},\\
\lambda_{2,t}
	&= \frac{1}{r} \psi_t' - \psi_t^3|\psi_t'|^{10} - \alpha\,\psi_t|\psi_t'|^6 - \frac{1}{100}|\psi_t'|^4\\
	&= -\frac{2}{1476225}\frac{t^{\frac{1}{3}}(2 \tilde P_4(t) - 492075\frac{r-r_0}{r})}{|r-r_0|^{\frac{4}{3}}},
\end{align*}
and where $\tilde P_4(t) = 6400\,t^4 + 32400\,\alpha\,t^2 + 729\,t$.

We next fix $\alpha = -\frac{36}{25}$. Then $\tilde P_4(-2) = -85682$, $\tilde P_4(-\frac{8}{5}) = - \frac{1966568}{25}$, $\tilde P_4(\frac{8}{5}) = -\frac{1908248}{25}$, $\tilde P_4(2) = -82766$ and so the equation $2\tilde P_4(t) + 164025 = 0$ has four roots $\tilde t_1, \ldots, \tilde t_4$ satisfying $\tilde t_1 < -2 < \tilde t_2 < -\frac{8}{5} < \frac{8}{5} < \tilde t_3 < 2 < \tilde t_4$. Note that $\lambda_{1,\tilde t_i}(r) = 0$ for $r \neq r_0$, $i = 1, \ldots, 4$. Also, we can find some small $\delta > 0$ such that
\[
\tilde t_i\,\lambda_{2,\tilde t_i}(r) > 0 \text{ for } r \neq r_0, |r - r_0| \leq \delta, i \in \{2, 3, 4\}.
\]
As $\tilde P_4(-3) = 96309 > 0$, we can also assume for $\tilde t_0 = -3$ that 
\[
\lambda_{1,\tilde t_0}(r) >0 \text{ and } \lambda_{2,\tilde t_0}(r) > 0\text{ for } r \neq r_0, |r - r_0| \leq \delta.
\]

We define, 
\begin{align*}
w(x) &= w(r) = \left\{\begin{array}{ll}
	\psi_{\tilde t_4}(r) & \text{ for } r_0 \leq r \leq r_0 + \delta,\\
	\psi_{\tilde t_2}(r) & \text{ for } r_0 -\delta \leq r < r_0,
\end{array}\right.\\
v(x) &= v(r) = \left\{\begin{array}{ll}
	\psi_{\tilde t_3}(r) & \text{ for } r_0 \leq r \leq r_0 + \delta,\\
	\psi_{\tilde t_0}(r) & \text{ for } r_0 -\delta \leq r < r_0.
\end{array}\right.
\end{align*}
It is readily seen that $w$ and $v$ are respectively a supersolution and a subsolution of \eqref{Eq:FpsiEq} for the operator $\tilde F[\psi] = \nabla^2\psi + \tilde L(\psi, \nabla\psi)$ in $\Omega = \{|r - r_0| < \delta\}$, $w \geq v$ in $\Omega$ and $\{w = v\} = \{r = r_0\}$. 

Now we proceed to modify $\tilde L$ to our desired $L$ as $\tilde L$ is not non-decreasing in $s$. We note that, as $|\tilde t_2| > \frac{8}{5}$ and $|\tilde t_3| > \frac{8}{5}$, $(w(x), \nabla w(x))$ and $(v(x),\nabla v(x))$ belong to the set
\[
N := \{(s,p) \in \RR \times \RR^n: s\,|p|^2 \in \RR \setminus (-\frac{32}{45},\frac{32}{45})\} \text{ for all } x \in \Omega \setminus \{r = r_0\}.
\]
As $\partial_s \tilde L(s,p) = -|p|^6(3s^2|p|^4 + \alpha)I = -|p|^6(3s^2|p|^4 - \frac{36}{25})I$, we see that $\tilde L$ is non-increasing in $s$ for $(s,p) \in N$.

Next, note that, for a fixed $p \neq 0$, 
\[
\tilde L(-\frac{32}{45}|p|^{-2}, p) = -\frac{245821}{364500}|p|^4I < 0 < \frac{238531}{364500}|p|^4I = \tilde L(\frac{32}{45}|p|^{-2}, p).
\]
Therefore, there exists a smooth function $L: \RR \times \RR^n \rightarrow \mathcal{S}^{n \times n}$ which is non-increasing in $s$ such that $L \equiv \tilde L$ in $N$ (e.g. by smoothly interpolating in $s$ the values of $\tilde L$ on the boundary of $N$). Then $w$ and $v$ are also a supersolution and a subsolution of \eqref{Eq:FpsiEq} for the operator $F[\psi] = \nabla^2\psi + L(\psi, \nabla\psi)$ in $\Omega$. This completes the proof.
\end{proof}

\section{Perron's method}\label{Sec:Perron}

We begin with the 
\begin{proof}[Proof of Theorem \ref{thm:Uniq}]
The conclusion a direct consequence of Theorem \ref{thm:CPQuad}(b).
\end{proof}

In the rest of this section, we prove Theorem \ref{thm:Perron}. We introduce some notations. For $O\subset \RR^n$,  $\xi: O\to [-\infty, +\infty]$,
let
$$
\xi^*(x):=\lim_{r\to 0^+}
\sup \{\xi(y)\ |\ y\in O, |y-x|<r\},
$$
$$
\xi_*(x):=\lim_{r\to 0^+}
\inf \{\xi(y)\ |\ y\in O, |y-x|<r\}.
$$
It is easy to see that, if $\xi^*(x) < +\infty$ for all $x \in O$, then $\xi^* \in USC(O)$. Likewise, if $\xi_*(x) > -\infty$ for all $x \in O$, then $\xi_* \in LSC(O)$.

$\xi^*$ is called the upper semicontinuous envelope of $\xi$,
it is the smallest upper semicontinuous
function satisfying $\xi\le \xi^*$ in $O$.
Similarly,
$\xi_*$ is called the lower semicontinuous envelope of $\xi$, it is the
largest
lower semicontinuous function satisfying $\xi\ge \xi_*$ in $O$.

Note that, for any constant $c$, $F[c] = 0 \in \partial U$. Thus, replacing $v$ by $\max(v,c)$ with some $c < \inf_{\partial\Omega} w$ and $w$ by $\min(w,c')$ with some $c' > \sup_{\partial\Omega} v$ if necessary, we can assume that
\[
-\infty < \inf_{\bar\Omega} v \leq \sup_{\bar\Omega} w < +\infty.
\]
Here we have used the fact that the maximum of two subsolutions is a subsolution and the minimum of two supersolutions is a supersolution.

Note that by hypotheses, $w\ge v$ in $\Omega$. Define
\begin{eqnarray}
u(x):=
&&
\inf\{\xi(x)\ |\
v\le \xi\le w\ \mbox{in}\ \overline \Omega,
\xi=v=w\ \mbox{on}\ \partial \Omega,\nonumber \\
&& \qquad \qquad \xi\in LSC(\overline \Omega),\
\xi\ \mbox{is a supersolution of \eqref{Eq:FpsiEq} in}\ \Omega\}.
\label{2new}
\end{eqnarray}

Clearly
$$
\inf_{ \overline \Omega} u
\ge \inf_{ \overline \Omega} v>-\infty.
$$

We will prove that the above defined $u$ satisfies the requirement of Theorem \ref{thm:Perron}.

\begin{lem}\label{lemC5-1new}
Let $O\subset \RR^n$ be an open set, $L: O \times \RR \times \RR^n \rightarrow \mathcal{S}^{n \times n}$ be continuous, $F$ be given by \eqref{Eq:FDef}, and
let ${\cal F}$ be a family of supersolutions of
\eqref{Eq:FpsiEq} in $O$.
Let
$$
\eta(x):=\inf\{\xi(x)\ |\
\xi\in {\cal F}\},
\ \ \ x\in O.
$$
Assume that $\eta_*(x)>-\infty\ \forall\ x\in O$.
Then $\eta_*$ is a supersolution of \eqref{Eq:FpsiEq} in $O$.
\end{lem}

\begin{proof}
Suppose for some $x\in O$ that there exists a polynomial $P$ of the form
$$
P(y):=a+p\cdot (y-x)+\frac 12 (y-x)^t M(y-x),
$$
with $a \in \RR$, $p\in \RR^n$, $M\in {\cal S}^{n\times n}$, such that,
for some $\epsilon>0$,
\begin{equation}
P(x)=\eta_*(x) \text{ and } P(y)\le \eta_*(y)\ \ \forall\ |y-x|<\epsilon.
\label{C6-1new}
\end{equation}
We will show that
\begin{equation}
F[P](x)\in {\cal S}^{n\times n} \setminus  U.
\label{C6-2new}
\end{equation}
It is standard that this implies that $\eta_*$ is a supersolution of \eqref{Eq:FpsiEq} in the sense of Definition \ref{Def:ViscositySolution}.

By the definition of $\eta_*$, there exists
$r_i\to 0^+$, $|x_i-x|<r_i$ such that
$$
\inf_{B_{r_i}(x)} \eta \leq \eta(x_i) \leq \inf_{B_{r_i}(x)} \eta + \frac{1}{i} \le \eta_*(x)  + \frac{1}{i} \text{ and }
 \eta(x_i)\to \eta_*(x).
$$
Moreover,
 there exists
$\xi_i\in {\cal F}$, such that
$\xi_i\ge \eta \ge \eta_*$ and
$$
0\le \xi_i(x_i)-\eta(x_i)<\frac 1i.
$$
We see from the above that
$$
\xi_i\ge \eta\ge \eta_*\ge
P \quad \mbox{in}\ B_\epsilon(x),
$$
and
$$
\xi_i(x_i)\to \eta_*(x)=P(x).
$$

For every $0<2\delta<\min\{\epsilon,
dist(x, \partial O)\}$,
consider
$$
P_\delta(y):= P(y)-\delta|y-x|^2.
$$
Then
$$
\xi_i\ge P_\delta\ \  \mbox{in}\ B_\epsilon(x),
\quad 
\xi_i\ge P_\delta+ \delta^3\ \  \mbox{in}\ B_\epsilon(x)\setminus
B_\delta(x), \text{ and }
\xi_i(x_i)-P_\delta(x_i)
\to 0.
$$

It follows that
 there exists $\beta_i=\circ(1) \ge 0$ and $x_i^*\in B_\delta(x)$ such that
\begin{equation}\label{C7-0new}
\xi_i(y)\ge P_\delta(y) + \beta_i,\ \ \mbox{in}\ B_\epsilon(x),
 \qquad
\xi_i(x_i^*)=P_\delta(x_i^*) + \beta_i.
\end{equation}
As $\xi_i$ is also a supersolution of \eqref{Eq:FpsiEq} in $O$.  Thus,
\begin{equation}
F[P_\delta + \beta_i](x_i^*)
\in {\cal S}^{n\times n}\setminus 
 U.
\label{C7-1new}
\end{equation}

\noindent{\bf Claim.}\  $x_i^*\to x$.

\medskip

Indeed, after passing to a subsequence,
$x_i^*\to \bar x$,  for some $\bar x$ satisfying
$|\bar x-x|\le \delta.$
By \eqref{C7-0new} and the definition of $\eta$ and $\eta_*$,
$$
\eta_*(x_i^*) - \beta_i \le \xi_i(x_i^*) - \beta_i=P_\delta(x_i^*).
$$
Sending $i$ to infinity in the above,
and using the lower-semicontinuity property of $\eta_*$,
we have
$
\eta_*(\bar x) \le P_\delta(\bar x)=P(\bar x)-\delta|\bar x-x|^2.
$
On the other hand, $P(\bar x)\le \eta_*(\bar x)$ according to
\eqref{C6-1new}.
Thus $\bar x=x$, and the claim is proved.

\medskip

With the convergence of $x_i^*$ to $x$ and of $\beta_i$ to $0$, sending $\delta$ to $0$ and $i$ to $\infty$ in
\eqref{C7-1new} yields
\eqref{C6-2new}.
Lemma \ref{lemC5-1new} is established.
\end{proof}

\bigskip

\begin{proof}[Proof of Theorem \ref{thm:Perron}]
We know that
\begin{equation}
\max(v,u_*)\le u\le \min(u^*, w),
\qquad\mbox{in}\ \overline \Omega,
\label{C9-1new}
\end{equation}
where $u$ is defined by \eqref{2new}.  Clearly,
\begin{equation}
v= u_*= u=
u^*= w,
\qquad\mbox{on}\ \partial \Omega,
\label{C9-2new}
\end{equation}
By Lemma \ref{lemC5-1new},
$u_*$ is a supersolution of \eqref{Eq:FpsiEq} in $\Omega$. By the comparison principle Theorem \ref{thm:CPQuad}(ii), $u_* \geq v$. Hence, by the definition of $u$, $u\le u_*$ 
in $\overline \Omega$.
Thus $u=u_*$ in $\overline \Omega$, and $u$ is a supersolution of \eqref{Eq:FpsiEq} 
in $\Omega$.

Note that
\[
\sup_{\bar\Omega} u^* \leq \sup_{\bar\Omega} w < +\infty.
\]
\medskip

\noindent{\bf Claim.}\
$u^*$ is a subsolution of \eqref{Eq:FpsiEq}  in $\Omega$.

\medskip

To prove this claim, we follow Ishii's argument (\cite{Ishii89-CPAM}). Indeed, if the claim does not hold, there exist $x\in \Omega$ and some quadratic polynomial
$$
P(y)=a+p\cdot (y-x)+\frac 12 (y-x)^t M (y-x),
$$
with $a \in \RR$, $p\in \RR^n$, $M\in {\cal S}^{n\times n}$, such that
for some $\bar \epsilon>0$ 
\begin{equation}
P(y)\ge u^*(y)\ \ \mbox{for}\ y\in
 B_{\bar\epsilon}(x),\qquad
P(x)=u^*(x),
\label{C10-1new}
\end{equation}
but
\begin{equation}
F[P](x)\in {\cal S}^{ n\times n}\setminus \overline U.
\label{C10-2new}
\end{equation}

Since ${\cal S}^{ n\times n}\setminus \overline U$ 
is open, there exists $0<2\bar\delta<
\min\{\bar \epsilon^2,  1\}$ such that
for all $0<\delta<\bar\delta$, the function
$$
P_\delta(y):=P(y)+\delta|y-x|^2 -\delta^2
$$
satisfies
\begin{equation}
P_\delta(x)=P(x)-\delta^2<u^*(x),
\label{C11-0new}
\end{equation}
and
\begin{equation}
F[P_\delta](y)\in {\cal S}^{ n\times n}\setminus \overline  U,\qquad
\forall\ |y-x|<\delta^{1/9}.
\label{C10-3new}
\end{equation}
Clearly,
\begin{equation}
P_\delta(y)>P(y),\qquad
\forall \ |y-x|\ge  \delta^{1/5}.
\label{C11-1new}
\end{equation}

Define
$$
\hat u(y):=
\left\{
\begin{array}{lr}
\displaystyle{
\min\{u(y), P_\delta(y)\},
}&
\mbox{if}\ |y-x|<\delta^{1/5},\\
u(y), &
\mbox{if}\ |y-x|\ge \delta^{1/5}.
\end{array}
\right.
$$
By \eqref{C10-3new},
$P_\delta$ is a supersolution of \eqref{Eq:FpsiEq}  in $\{y: |y-x|<\delta^{1/9}\}$.
By \eqref{C11-1new}, and using $P\ge u^*\ge u$, we have
$$
\hat u(y)=u(y)=
\min\{u(y), P_\delta(y)\},\qquad
\delta^{1/5}\le
|y-x|\le \delta^{1/6}.
$$
It follows that $\hat u$, being the minimum
 of two supersolutions,
 is a supersolution of \eqref{Eq:FpsiEq}  in $\Omega$, and,
because of the definition of $u$,
\begin{equation}
u\le \hat u\qquad\mbox{in}\ \Omega.
\label{C12-1new}
\end{equation}
On the other hand we see from
  \eqref{C11-0new},   the definition of $\hat u$ and \eqref{C12-1new}
that there exists $\epsilon\in (0, \delta^{1/5})$ such that
$$
u(y)\le
\hat u(y)\le P_\delta(y)<u^*(x)-\epsilon,\qquad
\forall\ |y-x|<\epsilon.
$$
Thus
$$
u^*(x)
=\lim_{r\to 0^+}
\sup\{u(y)\ |\ |y-x|<r\}
\le u^*(x)-\epsilon,
$$
a contradition.
The claim is proved, i.e. $u^*$ is a subsolution of \eqref{Eq:FpsiEq}  in $\Omega$. 

\bigskip

Now we have proved that $u_* = u$ and $u^*$ are respectively
supersolution and subsolution of \eqref{Eq:FpsiEq}  in $\Omega$, and $u_*=u^*$
on $\partial \Omega$. By the comparison principle Theorem \ref{thm:CPQuad}(ii), $u^* \leq u$ in $\Omega$ and so $u = u_* = u^*$ is a solution of \eqref{Eq:FpsiEq}.
\end{proof}

To conclude the section, let us remark that:

\begin{rem}
The conclusion of Theorem \ref{thm:Perron} is still valid for more general $(F,U)$ as in Theorem \ref{thm:CPQuad}, or Theorem \ref{thm:CPNUE} or Theorem \ref{thm:CPNUECat} provided that the function $L(x,s,p)$ is independent of $s$ and 
\[
-\infty < \inf_{\bar\Omega} v \leq \sup_{\bar\Omega} w < +\infty.
\]
\end{rem}

\section{Lipschitz regularity of viscosity solutions}\label{Sec:LipRegVS}

In this section we prove Theorem \ref{thm:regularity}, as an application of the comparison principle Theorem \ref{thm:CPQuad}. We also consider some mild generalization regarding Lipschitz regularity of viscosity solutions for operator of the form \eqref{Eq:FDef}.

\begin{proof}[Proof of Theorem \ref{thm:regularity}]
Without loss of generality, we may assume that $\Omega=B(0,1)$ and we only need to prove that $u$ is Lipschitz continuous on $\overline{B(0,\frac{1}{2})}$.

For any $x\in\overline{B(0,\frac{1}{2})}$, $0<\lambda\leq R:=\frac{1}{4}\left[\frac{\sup\limits_{B(0,\frac{3}{4})}u}{\inf\limits_{B(0,\frac{3}{4})}u}\right]^{-\frac{1}{n-2}}$, we define $u_{x,\lambda}$, the Kelvin transform of $u$, as 
\begin{equation}
u_{x,\lambda}(y):=\frac{\lambda^{n-2}}{|y-x|^{n-2}}u(x+\frac{\lambda^{2}(y-x)}{|y-x|^{2}}),\quad\forall y\in\overline{B(0,\frac{3}{4})\setminus B(x,\lambda)}.\label{jgd}
\end{equation}
For any $y\in\partial B(0,\frac{3}{4})$, we have 
\begin{equation*}
u_{x,\lambda}(y)\leq(4R)^{n-2}\sup\limits_{B(0,\frac{3}{4})}u=\inf\limits_{B(0,\frac{3}{4})}u\leq u(y).
\end{equation*}
Also, we know that 
\begin{equation*}
\lambda(A^{u_{x,\lambda}})\in \partial\Gamma,\quad\mbox{ in }B(0,\frac{3}{4})\setminus \overline{B(x,\lambda)},\quad\mbox{ in the viscosity sense}.
\end{equation*}
Since $u_{x,\lambda}=u$ on $\partial B(x,\lambda)$, by applying the comparison principle Theorem \ref{thm:CPQuad}(b) with $\Omega=B(0,\frac{3}{4})\setminus \overline{B(x,\lambda)}$, $U=\{M\in\Snn:\lambda(M)\in\Gamma\}$, $F[\psi] = A[\psi]$, $w={-\frac{2}{n-2}}\ln u_{x,\lambda}$ and $v={-\frac{2}{n-2}} \ln u$, we have  
\begin{equation}
u_{x,\lambda}\leq u\mbox{ in }B(0,\frac{3}{4})\setminus \overline{B(x,\lambda)}\mbox{ for any }0<\lambda\leq R,x\in\overline{B(0,\frac{1}{2})}.\label{lashi}
\end{equation}

By \cite[Lemma 2]{LiNg-arxiv}, \eqref{lashi} implies that $u$ is Lipschitz continuous on $\overline{B(0,\frac{1}{2})}$. This concludes the proof.
\end{proof}

As pointed out in the introduction, the above proof of Theorem \ref{thm:regularity} uses not only comparison principles but also conformal invariance property of the conformal Hessian. For general operators of the form \eqref{Eq:FDef}, one does not expect a purely local regularity like that in Theorem \ref{thm:regularity} to hold, as illustrated by the following example.

\begin{example}
Let $U$ be the set of symmetric $n \times n$ matrices $M$ with $M_{11} > 0$, and $L \equiv 0$. The equation $F[\psi] \in \partial U$ becomes
\[
\partial_{x_1}^2 \psi = 0.
\]
Then, the comparison principle holds (by considering the restriction of $\psi$ to each line parallel to the $x_1$-axis). Nevertheless, for any continuous function $f: \RR^{n-1} \rightarrow \RR$,  $\psi(x_1,x_2, \ldots, x_n) = f(x_2,\ldots, x_n)$ is a viscosity solution of $F[\psi] \in \partial U$, and clearly, the regularity of $\psi$ (with respect to the $x_2, \ldots, x_n$ variables) is not better than that of $f$. 
\end{example}

Despite the above negative example, by a variant of the proof of Theorem \ref{thm:regularity} using translational invariance rather than conformal invariance, we have the following partial generalization:

\begin{cor}\label{prop:ObLipReg}
Assume $n \geq 2$. Let $\Omega\subset\mathbb{R}^{n}$ be a bounded open set, $(F,U)$ be as in Theorem \ref{thm:CPQuad} with constant $\alpha$ and $\beta$. Assume that $\psi \in C^0(\Omega)$ is a viscosity solution to \eqref{Eq:FpsiEq} in $\Omega$. If $\psi \in C^{0,1}(\overline{N \cap \Omega})$ for some open neighborhood of $\partial \Omega$, then $\psi \in C^{0,1}(\bar\Omega)$ and
\[
\sup_{\Omega} |\nabla\psi| \leq \sup_{N \cap \Omega} |\nabla\psi|.
\]
\end{cor}

Before giving a proof, we remark that, in general, the Lipschitz regularity of $\psi$ on $\partial\Omega$ does not ensure that the solution $\psi$ is Lipschitz continuous in $\bar\Omega$.
\begin{example}
Consider the equation 
\begin{equation}
F[\psi] = \nabla^2 \psi - |\nabla \psi|^{m}\,I \in \partial U
	\label{Eq:BLipCounterEx}
\end{equation}
where $m > 2$ and $U$ is the set of symmetric $n \times n$ matrices with at least one positive eigenvalue. (This equation can be written equivalently as
\[
\det (F[\psi]) = 0 \text{ and } F[\psi] \leq 0.)
\]
Then $\psi(x) = - (m-1)^{\frac{m-2}{m-1}} (m-2)^{-1}\,(|x|-1)^{\frac{m-2}{m-1}}$ is a solution to \eqref{Eq:BLipCounterEx} on $\Omega_a = \{1 < |x| < a\}$ for any $a > 1$. Clearly $\psi$ is constant on each component of the boundary $\partial \Omega_a$, but $\psi \notin C^{0,1}(\overline{\Omega_a})$.
\end{example}

\begin{proof}[Proof of Corollary \ref{prop:ObLipReg}]
Shrinking $\Omega$ and $N$ if necessary, we may assume that $\psi \in C^{0,1}(\bar N)$.

%By hypotheses and Theorems \ref{thm:CPQuad}, \ref{thm:CPNUE} and \ref{thm:CPNUECat}, the comparison principle holds for $(F,U)$.

We note that for any vector $e \in \RR^n$ and any constant $c \in \RR$, the function
\[
\psi_{e}(x) := \psi(x + e)
\]
satisfies $F[\psi_{e} + c] \in \partial U$ in $\Omega_e := \{x: x + e \in \Omega\}$ in the viscosity sense. Thus, by the comparison principle Theorem \ref{thm:CPQuad}(b),
\[
\psi \leq \psi_e + \max_{\partial (\Omega \cap \Omega_e)} (\psi - \psi_e) \text{ in } \Omega \cap \Omega_e.
\]
In particular, there is some $\delta > 0$ such that for $|e| < \delta$, we have $\partial (\Omega \cap \Omega_e) \subset \bar N$ and 
\[
\psi \leq \psi_e + \sup_N |\nabla\psi||e| \text{ in }  \Omega \cap \Omega_e.
\]
This implies the assertion.
\end{proof}

\noindent{\bf{\large Acknowledgments.}} Li is partially supported by NSF grant DMS-1501004. Wang is supported in part by the scholarship from China Scholarship Council under the Grant CSC No. 201406040131. 

%\bibliography{paris}{}

\begin{thebibliography}{10}

\bibitem{AmendolaGaliseVitolo13-DIE}
{\sc M.~E. Amendola, G.~Galise, and A.~Vitolo}, {\em Riesz capacity, maximum
  principle, and removable sets of fully nonlinear second-order elliptic
  operators}, Differential Integral Equations, 26 (2013), pp.~845--866.

\bibitem{BardiDaLio99-AM}
{\sc M.~Bardi and F.~Da~Lio}, {\em On the strong maximum principle for fully
  nonlinear degenerate elliptic equations}, Arch. Math. (Basel), 73 (1999),
  pp.~276--285.

\bibitem{HBDolcettaPorretaRosi15-JMPA}
{\sc H.~Berestycki, I.~Capuzzo~Dolcetta, A.~Porretta, and L.~Rossi}, {\em
  Maximum principle and generalized principal eigenvalue for degenerate
  elliptic operators}, J. Math. Pures Appl. (9), 103 (2015), pp.~1276--1293.

\bibitem{BirindelliDemengel04-AFSTM}
{\sc I.~Birindelli and F.~Demengel}, {\em Comparison principle and {L}iouville
  type results for singular fully nonlinear operators}, Ann. Fac. Sci. Toulouse
  Math. (6), 13 (2004), pp.~261--287.

\bibitem{BirindelliDemengel07-CPAA}
\leavevmode\vrule height 2pt depth -1.6pt width 23pt, {\em Eigenvalue, maximum
  principle and regularity for fully non linear homogeneous operators}, Commun.
  Pure Appl. Anal., 6 (2007), pp.~335--366.

\bibitem{CabreCaffBook}
{\sc L.~Caffarelli and X.~Cabr{\'e}}, {\em Fully nonlinear elliptic equations},
  vol.~43 of American Mathematical Society Colloquium Publications, American
  Mathematical Society, Providence, RI, 1995.

\bibitem{CafLiNir11}
{\sc L.~Caffarelli, Y.~Y. Li, and L.~Nirenberg}, {\em Some remarks on singular
  solutions of nonlinear elliptic equations {III}: viscosity solutions
  including parabolic operators}, Comm. Pure Appl. Math., 66 (2013),
  pp.~109--143.

\bibitem{C-N-S-Acta}
{\sc L.~Caffarelli, L.~Nirenberg, and J.~Spruck}, {\em The {D}irichlet problem
  for nonlinear second-order elliptic equations. {III}. {F}unctions of the
  eigenvalues of the {H}essian}, Acta Math., 155 (1985), pp.~261--301.

\bibitem{CGY02-AnnM}
{\sc S.-Y.~A. Chang, M.~J. Gursky, and P.~Yang}, {\em An equation of
  {M}onge-{A}mp\`ere type in conformal geometry, and four-manifolds of positive
  {R}icci curvature}, Ann. of Math. (2), 155 (2002{\noopsort{a}}),
  pp.~709--787.

\bibitem{Chen05}
{\sc S.-y.~S. Chen}, {\em Local estimates for some fully nonlinear elliptic
  equations}, Int. Math. Res. Not.,  (2005), pp.~3403--3425.

\bibitem{UserGuide}
{\sc M.~G. Crandall, H.~Ishii, and P.-L. Lions}, {\em User's guide to viscosity
  solutions of second order partial differential equations}, Bull. Amer. Math.
  Soc. (N.S.), 27 (1992), pp.~1--67.

\bibitem{DolcettaVitolo07-MC}
{\sc I.~C. Dolcetta and A.~Vitolo}, {\em On the maximum principle for viscosity
  solutions of fully nonlinear elliptic equations in general domains},
  Matematiche (Catania), 62 ({\noopsort{a}}2007), pp.~69--91.

\bibitem{DolcettaVitolo-preprint}
\leavevmode\vrule height 2pt depth -1.6pt width 23pt, {\em The weak maximum
  principle for degenerate elliptic operators in unbounded domains},
  ({\noopsort{b}}preprint).

\bibitem{GeWang06}
{\sc Y.~Ge and G.~Wang}, {\em On a fully nonlinear {Y}amabe problem}, Ann. Sci.
  \'Ecole Norm. Sup. (4), 39 (2006), pp.~569--598.

\bibitem{Gonzalez05}
{\sc M.~d.~M. Gonz{\'a}lez}, {\em Singular sets of a class of locally
  conformally flat manifolds}, Duke Math. J., 129 (2005), pp.~551--572.

\bibitem{GW03-IMRN}
{\sc P.~Guan and G.~Wang}, {\em Local estimates for a class of fully nonlinear
  equations arising from conformal geometry}, Int. Math. Res. Not.,  (2003),
  pp.~1413--1432.

\bibitem{GV07}
{\sc M.~J. Gursky and J.~A. Viaclovsky}, {\em Prescribing symmetric functions
  of the eigenvalues of the {R}icci tensor}, Ann. of Math. (2), 166 (2007),
  pp.~475--531.

\bibitem{HLT10}
{\sc Z.-C. Han, Y.~Y. Li, and E.~V. Teixeira}, {\em Asymptotic behavior of
  solutions to the {$\sigma_k$}-{Y}amabe equation near isolated singularities},
  Invent. Math., 182 (2010), pp.~635--684.

\bibitem{HartmanNirenberg}
{\sc P.~Hartman and L.~Nirenberg}, {\em On spherical image maps whose
  {J}acobians do not change sign}, Amer. J. Math., 81 (1959), pp.~901--920.

\bibitem{HarveyLawsonSurvey2013}
{\sc F.~R. Harvey and H.~B. Lawson, Jr.}, {\em Existence, uniqueness and
  removable singularities for nonlinear partial differential equations in
  geometry}, in Surveys in differential geometry. {G}eometry and topology,
  vol.~18 of Surv. Differ. Geom., Int. Press, Somerville, MA, 2013,
  pp.~103--156.

\bibitem{Ishii89-CPAM}
{\sc H.~Ishii}, {\em On uniqueness and existence of viscosity solutions of
  fully nonlinear second-order elliptic {PDE}s}, Comm. Pure Appl. Math., 42
  (1989), pp.~15--45.

\bibitem{IshiiLions90-JDE}
{\sc H.~Ishii and P.-L. Lions}, {\em Viscosity solutions of fully nonlinear
  second-order elliptic partial differential equations}, J. Differential
  Equations, 83 (1990), pp.~26--78.

\bibitem{Jensen88-ARMA}
{\sc R.~Jensen}, {\em The maximum principle for viscosity solutions of fully
  nonlinear second order partial differential equations}, Arch. Rational Mech.
  Anal., 101 (1988), pp.~1--27.

\bibitem{J}
{\sc Q.~Jin}, {\em Local {H}essian estimates for some conformally invariant
  fully nonlinear equations with boundary conditions}, Differential and
  Integral Equations, 20 (2007), pp.~121--132.

\bibitem{J-L-L}
{\sc Q.~Jin, A.~Li, and Y.~Y. Li}, {\em Estimates and existence results for a
  fully nonlinear {Y}amabe problem on manifolds with boundary}, Calc. Var.
  Partial Differential Equations, 28 (2007), pp.~509--543.

\bibitem{KawohlKutev98-AM}
{\sc B.~Kawohl and N.~Kutev}, {\em Strong maximum principle for semicontinuous
  viscosity solutions of nonlinear partial differential equations}, Arch. Math.
  (Basel), 70 (1998), pp.~470--478.

\bibitem{KawohlKutev00-FE}
\leavevmode\vrule height 2pt depth -1.6pt width 23pt, {\em Comparison principle
  and {L}ipschitz regularity for viscosity solutions of some classes of
  nonlinear partial differential equations}, Funkcial. Ekvac., 43 (2000),
  pp.~241--253.

\bibitem{KawohlKutev07-CPDE}
\leavevmode\vrule height 2pt depth -1.6pt width 23pt, {\em Comparison principle
  for viscosity solutions of fully nonlinear, degenerate elliptic equations},
  Comm. Partial Differential Equations, 32 (2007), pp.~1209--1224.

\bibitem{KoikeKosugi15-CPAA}
{\sc S.~Koike and T.~Kosugi}, {\em Remarks on the comparison principle for
  quasilinear {PDE} with no zeroth order terms}, Commun. Pure Appl. Anal., 14
  (2015), pp.~133--142.

\bibitem{KoikeLey11-JMAA}
{\sc S.~Koike and O.~Ley}, {\em Comparison principle for unbounded viscosity
  solutions of degenerate elliptic {PDE}s with gradient superlinear terms}, J.
  Math. Anal. Appl., 381 (2011), pp.~110--120.

\bibitem{LiLi03}
{\sc A.~Li and Y.~Y. Li}, {\em On some conformally invariant fully nonlinear
  equations}, Comm. Pure Appl. Math., 56 (2003), pp.~1416--1464.

\bibitem{LiLi05}
\leavevmode\vrule height 2pt depth -1.6pt width 23pt, {\em On some conformally
  invariant fully nonlinear equations. {II}. {L}iouville, {H}arnack and
  {Y}amabe}, Acta Math., 195 (2005), pp.~117--154.

\bibitem{Li09-CPAM}
{\sc Y.~Y. Li}, {\em Local gradient estimates of solutions to some conformally
  invariant fully nonlinear equations}, Comm. Pure Appl. Math., 62 (2009),
  pp.~1293--1326.
\newblock (C. R. Math. Acad. Sci. Paris 343 (2006), no. 4, 249--252).

\bibitem{LiNgBocher}
{\sc Y.~Y. Li and L.~Nguyen}, {\em Harnack inequalities and {B}\^ocher-type
  theorems for conformally invariant, fully nonlinear degenerate elliptic
  equations}, Comm. Pure Appl. Math., 67 (2014), pp.~1843--1876.

\bibitem{LiNgSymmetry}
\leavevmode\vrule height 2pt depth -1.6pt width 23pt, {\em Symmetry,
  quantitative {L}iouville theorems and analysis of large solutions of
  conformally invariant fully nonlinear elliptic equations},  (2016).
\newblock http://arxiv.org/abs/1604.06039.

\bibitem{LiNg-arxiv}
\leavevmode\vrule height 2pt depth -1.6pt width 23pt, {\em A fully nonlinear
  version of the {Y}amabe problem on locally conformally flat manifolds with
  umbilic boundary},  ({\noopsort{a}}2009).
\newblock http://arxiv.org/abs/0911.3366v1.

\bibitem{LiNir-misc}
{\sc Y.~Y. Li and L.~Nirenberg}, {\em A miscellany}, in Percorsi incrociati (in
  ricordo di Vittorio Cafagna), Collana Scientifica di Ateneo, Universita di
  Salerno, 2010, pp.~193--208.
\newblock http://arxiv.org/abs/0910.0323.

\bibitem{NadirashviliVladuts}
{\sc N.~Nadirashvili and S.~Vl\u{a}du\c{t}}, {\em Singular solutions of
  conformal {H}essian equation},  (2014).
\newblock http://arxiv.org/abs/1409.1454.

\bibitem{STW07}
{\sc W.-M. Sheng, N.~S. Trudinger, and X.-J. Wang}, {\em The {Y}amabe problem
  for higher order curvatures}, J. Differential Geom., 77 (2007), pp.~515--553.

\bibitem{Trudinger88-RMI}
{\sc N.~S. Trudinger}, {\em Comparison principles and pointwise estimates for
  viscosity solutions of nonlinear elliptic equations}, Rev. Mat.
  Iberoamericana, 4 (1988), pp.~453--468.

\bibitem{TW09}
{\sc N.~S. Trudinger and X.-J. Wang}, {\em On {H}arnack inequalities and
  singularities of admissible metrics in the {Y}amabe problem}, Calc. Var.
  Partial Differential Equations, 35 (2009), pp.~317--338.

\bibitem{Viac00-Duke}
{\sc J.~A. Viaclovsky}, {\em Conformal geometry, contact geometry, and the
  calculus of variations}, Duke Math. J., 101 (2000), pp.~283--316.

\bibitem{Viac02-CAG}
\leavevmode\vrule height 2pt depth -1.6pt width 23pt, {\em Estimates and
  existence results for some fully nonlinear elliptic equations on {R}iemannian
  manifolds}, Comm. Anal. Geom., 10 (2002), pp.~815--846.

\bibitem{Wang06}
{\sc X.-J. Wang}, {\em A priori estimates and existence for a class of fully
  nonlinear elliptic equations in conformal geometry}, Chinese Ann. Math. Ser.
  B, 27 (2006), pp.~169--178.

\end{thebibliography}
%\bibliographystyle{siam}

\newcommand{\noopsort}[1]{}

\end{document}